\documentclass[reqno, 12pt]{amsart}

\usepackage{amsmath,amssymb,mathrsfs,mathtools,verbatim} 
\usepackage[backrefs]{amsrefs}
\usepackage[protrusion=true]{microtype}
\usepackage[T1]{fontenc}
\usepackage[margin=1in]{geometry}
\usepackage[onehalfspacing]{setspace}
\usepackage[toc,page]{appendix}
\usepackage{txfonts}
\usepackage[colorlinks,pagebackref,hypertexnames=false,bookmarks=false]{hyperref}
\numberwithin{equation}{section}							
\usepackage[nameinlink,noabbrev]{cleveref}
\usepackage{autonum}										

\let\originalleft\left
\let\originalright\right
\renewcommand{\left}{\mathopen{}\mathclose\bgroup\originalleft}
\renewcommand{\right}{\aftergroup\egroup\originalright}

\makeatletter
\renewcommand*{\eqref}[1]{\hyperref[{#1}]{\textup{\tagform@{\ref*{#1}}}}}		
\makeatother

\newcommand{\rG}{\rm{G}}

\newcommand{\bA}{{\bf A}}

\newcommand{\bD}{{\bf D}}

\newcommand{\bG}{{\bf G}}

\newcommand{\cA}{\mathcal{A}}

\newcommand{\cC}{\mathcal{C}}
\newcommand{\cD}{\mathcal{D}}
\newcommand{\cE}{\mathcal{E}}

\newcommand{\cG}{\mathcal{G}}
\newcommand{\cH}{\mathcal{H}}
\newcommand{\cI}{\mathcal{I}}

\newcommand{\cL}{\mathcal{L}}
\newcommand{\cM}{\mathcal{M}}

\newcommand{\cR}{\mathcal{R}}

\newcommand{\cV}{\mathcal{V}}

\newcommand{\SO}{\rm{SO}}

\newcommand{\SU}{\rm{SU}}

\newcommand{\Ad}{\mathrm{Ad}}
\newcommand{\Aut}{\mathrm{Aut}}

\renewcommand{\epsilon}{\varepsilon}

\newcommand{\Lie}{\mathrm{Lie}}

\newcommand{\ad}{\mathrm{ad}}

\newcommand{\diag}{\mathrm{diag}}

\newcommand{\id}{\mathrm{id}}

\renewcommand{\Im}{\mathop{\mathrm{Im}}}

\renewcommand{\Re}{\mathop{\mathrm{Re}}}

\newcommand{\vol}{\mathrm{vol}}

\newcommand{\Fix}{\mathrm{Fix}}

\def\ad{\mathrm{ad}}

\def\Aut{\mathrm{Aut}}

\def\vol{\mathrm{vol}}

\def\Im{\mathrm{Im}}

\def\Lie{\mathrm{Lie}}

\def\Re{\mathrm{Re}}

\def\SO{\mathrm{SO}}

\def\SU{\mathrm{SU}}

\def\g{\mathfrak{g}}

\def\<{\mathopen{}\left<}
\def\>{\right>\mathclose{}}
\def\({\mathopen{}\left(}
\def\){\right)\mathclose{}}

\newtheorem{theorem}{Theorem}[section]
\newtheorem{Mtheorem}{Main Theorem}
\newtheorem*{acknowledgment}{Acknowledgment}

\newtheorem{corollary}[theorem]{Corollary}

\newtheorem{definition}[theorem]{Definition}
\newtheorem{example}[theorem]{Example}

\newtheorem{lemma}[theorem]{Lemma}

\newtheorem{proposition}[theorem]{Proposition}
\newtheorem{remark}[theorem]{Remark}

\newtheorem{hypoth}[theorem]{Hypothesis}
\crefname{theorem}{Theorem}{Theorems}						
\creflabelformat{theorem}{#2{#1}#3}							
\crefname{Mtheorem}{Main Theorem}{Main Theorems}			
\creflabelformat{Mtheorem}{#2{#1}#3}						
\crefname{lemma}{Lemma}{Lemmata}							
\creflabelformat{lemma}{#2{#1}#3}							
\crefname{corollary}{Corollary}{Corollaries}				
\creflabelformat{corollary}{#2{#1}#3}						
\crefname{proposition}{Proposition}{Propositions}			
\creflabelformat{proposition}{#2{#1}#3}						
\crefname{ineq}{inequality}{inequalities}					
\creflabelformat{ineq}{#2{\upshape(#1)}#3}					
\crefname{cond}{condition}{conditions}						
\creflabelformat{cond}{#2{\upshape(#1)}#3}					
\crefname{hypoth}{Hypothesis}{Hypotheses}					
\creflabelformat{hypoth}{#2{#1}#3}							
\crefname{def}{Definition}{Definitions}						
\creflabelformat{def}{#2{#1}#3}								
\crefname{appsec}{Appendix}{Appendices}

\title{The Haydys monopole equation}

\author{\'Akos Nagy}
\address[\'Akos Nagy]{Duke University, Durham, NC, USA}
\urladdr{\href{https://akosnagy.com}{akosnagy.com}}
\email{\href{mailto:contact@akosnagy.com}{contact@akosnagy.com}}

\author{Gon\c{c}alo Oliveira}
\address[Gon\c{c}alo Oliveira]{Universidade Federal Fluminense IME--GMA, Niter\'oi, Brazil}
\urladdr{\href{https://sites.google.com/view/goncalo-oliveira-math-webpage/home}{sites.google.com/view/goncalo-oliveira-math-webpage/home}}
\email{\href{mailto:galato97@gmail.com}{galato97@gmail.com}}

\date{\today}
\keywords{Monopoles, moduli spaces, Haydys equation, hyperk\"ahler manifolds}
\subjclass[2010]{53C07, 58D27, 58E15, 70S15}

\calclayout
\hypersetup{
	pdffitwindow	= true,
	unicode			= true,
	pdffitwindow	= true,
	pdftoolbar		= false,
	pdfmenubar		= false,
	pdfstartview	= {FitH},
	pdftitle 		= {The Haydys monopole equation},
	pdfkeywords		= {Monopoles, moduli spaces, Haydys equation, hyperk\"ahler manifolds},
	pdfauthor		= {\'Akos Nagy and Gon\c{c}alo Oliveira},
	colorlinks		= true,
	linkcolor		= black,
	citecolor		= black,
	filecolor		= blue,
	urlcolor		= blue
}

\begin{document}

\begin{abstract}
	We study complexified Bogomolny monopoles using the complex linear extension of the Hodge star operator; these monopoles can be interpreted as solutions to the Bogomolny equation with a complex gauge group. Alternatively, these equations can be obtained from dimensional reduction of the Haydys instanton equations to three dimensions, thus we call them Haydys monopoles.\\
	We find that (under mild hypotheses) the smooth locus of the moduli space of finite energy Haydys monopoles on $\mathbb{R}^3$ is a K\"ahler manifold containing the ordinary Bogomolny moduli space as a minimal Lagrangian submanifold---an A-brane. Moreover, using a gluing construction we construct an open neighborhood of this submanifold modeled on a neighborhood of the zero section in the tangent bundle to the Bogomolny moduli space. This is analogous to the case of Higgs bundles over a Riemann surface, where the (co)tangent bundle of holomorphic bundles canonically embeds into the Hitchin moduli space.\\
	These results contrast immensely with the case of finite energy Kapustin--Witten monopoles for which we have shown a vanishing theorem in \cite{NO21}. 
\end{abstract}

\maketitle

\tableofcontents

\section{Introduction and main results}

\subsection{Preparation and motivation}
Let $(M, g)$ be a Riemannian 3-manifold, and $\Lambda^* M$ its exterior algebra bundle. For any orthogonal vector bundle $E \rightarrow M$ the Hodge star operator extends to $E$-valued differential forms yielding a map $\ast : \Lambda^* M \otimes E \rightarrow \Lambda^{3 - *} M \otimes E$. Fix a principal $\rG$-bundle $P \rightarrow M$, where $\rG$ is a compact Lie group. A smooth pair $(\nabla, \Phi)$ consisting of a $\rG$-connection on $P$ and a section of $\g_P = \ad (P)$ (equipped with a $\rG$-invariant inner product), is called a {\em Bogomolny monopole} if 
\begin{equation}
	\ast \! F_\nabla = d_\nabla \Phi.  \label{eq:BPSeq}
\end{equation}
In the situation when $M = \mathbb{R}^3$ equipped with the Euclidean metric and $\rG = \SU (2)$ several things are known about solutions to this equation. For instance, up to the action of the automorphisms of $P$, the (finite energy) Bogomolny monopoles form a smooth noncompact moduli space. This can be equipped with the canonical $L^2$-metric which turns out to be complete and hyperk\"ahler. For higher rank structure groups, for instance when $\rG = \SU (N)$ with $N > 2$, less is know, but in many cases (cf.  \Cref{hypoth:maxsymbr} and \Cref{rem:maxsymbr} later), analogous results hold. In particular, when the moduli space is smooth at a Bogomolny monopole $m = (\nabla, \Phi)$ any (gauge fixed) tangent vector $v = (a, \Psi)$ at $m$ satisfies the linearized Bogomolny monopole equations:
\begin{subequations}
\begin{align}\label{eq:BPS_tangent_eq_1} 
	\ast d_{\nabla} a - d_{\nabla} \Psi - [a, \Phi]	& = 0,   \\ \label{eq:BPS_tangent_eq_2}
	d_{\nabla}^{\ast} a + [\Psi, \Phi] 	& = 0 , 
\end{align}
\end{subequations}
with the second equation arising from requiring the tangent vector to be orthogonal to the slice cut out by the action of the gauge group at $m$. Moreover, the formal $L^2$ dual equations
\begin{subequations}
\begin{align}\label{eq:BPS_dual_tangent_eq_1} 
	\ast d_{\nabla} a + d_{\nabla} \Psi + [a, \Phi]	& = 0,   \\ \label{eq:BPS_dual_tangent_eq_2}
	d_{\nabla}^* a + [\Psi, \Phi] 	& = 0,  
\end{align}
\end{subequations}
have no solutions in $L^2$, at least under certain standard hypotheses; again, see \Cref{hypoth:maxsymbr}. This is in fact the reason why the Implicit Function Theorem can be used to show smoothness of the moduli space of finite energy Bogomolny monopoles.

Now we ``complexify'' the Bogomolny \cref{eq:BPSeq} by considering $\Lambda_{\mathbb{C}}^* M = \Lambda^* M \otimes_{\mathbb{R}} \mathbb{C}$, $\rG_{\mathbb{C}}$ the complexification of $\rG$, and $P_{\mathbb{C}} = P \times_{\rG} \rG_{\mathbb{C}}$ which is the principal $\rG_{\mathbb{C}}$-bundle associated with the standard conjugation action of $\rG$ on $\rG_{\mathbb{C}}$. The Hodge star operator $\ast$ may now be extended in two inequivalent ways to $\Lambda_{\mathbb{C}}^* M \otimes \g_P \simeq \Lambda^* M \otimes \g_{P_{\mathbb{C}}}$. This may be either as a complex linear operator, which we still denote $\ast$, or as a conjugate linear one, which we denote by $\overline{\ast}$.\\
Depending on which such extension one uses we obtain two different complex monopole equations. In this second paper we only consider one of these which is made using $\ast$. Let $(\bA, \Upsilon)$ respectively be a connection on $P_{\mathbb{C}}$ and a section of $\g_{P_{\mathbb{C}}}$. Then, we have the following complex monopole equation
\begin{equation}\label{eq:Haydys_mono}
	\ast \! F_{\bA} = d_{\bA} \Upsilon,  
\end{equation}

\begin{remark}
	The equation obtained using $\overline{\ast}$ instead is given by
	\begin{equation}\label{eq:KW_mono}
		\overline{\ast} F_{\bA} = d_{\bA} \Upsilon,
	\end{equation}
	and studied in the second paper in this series \cite{NO21}.
\end{remark}

Let the real gauge group be denoted by $\cG = \Aut (P)$ and the complex one by $\cG_{\mathbb{C}} = \Aut (P_\mathbb{C})$. Both complex monopole \cref{eq:Haydys_mono,eq:KW_mono} are invariant under the usual action of $\cG_{\mathbb{C}}$. In order to work only modulo the action of $\cG$ we proceed as follows. Observe that $\bA$ can be uniquely written as $\bA = \nabla + i a$, with $\nabla$ a connection on $P$ and $a \in \Omega^1 (M, \g_P)$. Similarly $\Upsilon = \tfrac{1}{\sqrt{2}} (\Phi + i \Psi)$, with $\Phi, \Psi \in \Omega^0 (M, \g_P)$. A standard procedure in gauge theory to ``break down the gauge symmetry'' is to include an extra equation obtained by formally working orthogonally to the gauge action, such as a Coulomb gauges as used in \cite{Uhlenbeck1982a}. In favorable cases this has the effect of making the resulting gauge theoretical equations elliptic. In the situation at hand, instead of breaking the full symmetry we may first ``break it down'' to that of the real gauge group $\cG$ by imposing an extra equation of the form
\begin{equation}\label{eq:Coulomb}
	d_{\nabla}^* a = i [\Upsilon, \overline{\Upsilon}] \ \ \  \Leftrightarrow \ \ \ d_{\nabla}^{\ast} a + [\Psi, \Phi] = 0,  
\end{equation}
which is a Coulomb type gauge fixing condition. This makes the full system of PDE's elliptic.

\begin{remark}
	Alternatively, this extra equation may be motivated by comparison with the Kempf--Ness Theorem in finite dimensional situations. Indeed, \cref{eq:Coulomb} may be interpreted as a moment map equation for an Hamiltonian action of $\cG$ on the space of quadruples $(\nabla , a, \Psi , \Phi)$ equipped with a natural $L^2$-symplectic structure.
\end{remark}

One other point of interest in \cref{eq:Haydys_mono,eq:KW_mono,eq:Coulomb} is that they may be obtained from dimensional reductions of the instanton equations of Haydys (cf.  \cite{H15}) and Kapustin--Witten (cf.  \cite{KW07}) respectively. In this paper we focus on the first of these. For more on the later we refer the reader to our parallel article \cite{NO21}. While the Kapustin--Witten equations, and their dimensional reductions, have been attracting an increasing amount of interest from the mathematical community (see, for example, \cites{W12,W18,T17,T18,T19,MW14,HM17,HW19}) the Haydys equation has remained less explored. However, it was pointed out in \cite{H15} that its moduli space carries interesting geometric structures which have a shadow in the dimensional reduction that we consider in \Cref{sec:Geometric_Strucutres}. For completeness and motivation we included the corresponding computations in \cref{sec:dimred}.

In order to establish notation, recall that the wedge product of two $\g_{P^X}$-valued forms $a = \sum_{|I| = p} a_I dx^I$ and $b = \sum_{|J| = q} b_J dx^J$ is given by
\begin{equation}
	[a \wedge b] = \sum\limits_{\substack{|I| = p \\ |J| = q}} [a_I, b_J] dx^I \wedge dx^J,
\end{equation}
and satisfies $[a \wedge b] = (-1)^{pq+1} [b \wedge a]$. Using this, a simple computation shows that \cref{eq:Haydys_mono,eq:Coulomb} are equivalent to
\begin{subequations}
\begin{align} \label{eq:Haydys_Mono_1}
	\ast F_{\nabla} - d_{\nabla} \Phi - \tfrac{1}{2} \ast [a \wedge a] + [a, \Psi] 	&= 0,   \\ \label{eq:Haydys_Mono_2} 
	\ast d_{\nabla} a - d_{\nabla} \Psi - [a,\Phi] 	&= 0,  \\ \label{eq:Haydys_Mono_3}
	d_{\nabla}^{\ast} a + [\Psi,\Phi] 		&= 0.
\end{align}
\end{subequations}

\begin{remark}
A similar computation shows that \cref{eq:KW_mono,eq:Coulomb} are equivalent to
\begin{subequations}
	\begin{align}
	\ast F_{\nabla} - d_{\nabla} \Phi - \tfrac{1}{2} \ast [a \wedge a] + [a, \Psi]	&= 0,  \label{eq:KW_Mono_1}  \\
	\ast d_{\nabla} a + d_{\nabla} \Psi + [a, \Phi] 	&= 0,  \label{eq:KW_Mono_2}  \\
	d_{\nabla}^* a + [\Psi, \Phi] 	&= 0.  \label{eq:KW_Mono_3}
	\end{align}
\end{subequations}
Again, we point out that we studied these equations in the parallel article \cite{NO21}.
\end{remark}

Given that, as mentioned in the previous paragraph, these equations are obtained from dimensional reduction of the Haydys instanton equation, we name \cref{eq:Haydys_Mono_1,eq:Haydys_Mono_2,eq:Haydys_Mono_3} the {\em Haydys monopole equations} and their solutions {\em Haydys monopoles}. In the same way, we call \cref{eq:KW_Mono_1,eq:KW_Mono_2,eq:KW_Mono_3} {\em Kapustin--Witten monopole equations} and their solutions {\em Kapustin--Witten monopoles}. Observe that both these sets of gauge theoretic equations with gauge group $\cG$ (rather than $\cG_{\mathbb{C}}$), are elliptic modulo its action as already explained above.

Furthermore, notice that the \cref{eq:Haydys_Mono_1,eq:KW_Mono_1} are the same, and can be seen as a quadratic (but algebraic) perturbation of the Bogomolny monopole \cref{eq:BPSeq}. As for the second and third Haydys monopole \cref{eq:Haydys_Mono_2,eq:Haydys_Mono_3}, these are exactly the tangent space \cref{eq:BPS_tangent_eq_1,eq:BPS_tangent_eq_2} for the Bogomolny moduli space. On the other hand, the second and third Kapustin--Witten monopole \cref{eq:KW_Mono_2,eq:KW_Mono_3} are dual \cref{eq:BPS_dual_tangent_eq_1,eq:BPS_dual_tangent_eq_2}.

We now introduce the relevant energy functional in this complex monopole setting. Denote by $\| \cdot \|$ the usual $L^2$ norm for sections of any bundle over $M$. Given a quadruple $(\nabla, \Phi, a, \Psi)$ as before, we define a Yang--Mills--Higgs type energy functional by
\begin{equation}\label{eq:YMH_energy}
\begin{aligned}
	\cE (\nabla, \Phi, a, \Psi) &= \| F_\nabla \|^2 + \| \nabla a \|^2 + \| \nabla \Phi \|^2 + \| \nabla \Psi \|^2 + \tfrac{1}{4} \| [a \wedge a] \|^2 \\
	&\quad + \| [a,\Phi] \|^2 + \| [a,\Psi] \|^2 + \| [\Psi,\Phi] \|^2 .
\end{aligned}
\end{equation}
We also point out, without proof, that up to an overall constant and for $M$ Ricci flat, the energy \eqref{eq:YMH_energy} is simply a sum of the $L^2$ norms of $F_{\nabla + i a}$ with $d_{\nabla + i a} (\Phi + i \Psi)$ on $M$.

\subsection{Main results}

Before stating our main results we recall some---by now classic---results on the moduli spaces of Bogomolny monopoles. Let $M = \mathbb{R}^3$ with the Euclidean metric, $\rG$ a compact Lie group, and denote by $\cM_B$ and $\cM_H$ the moduli spaces of Bogomolny and Haydys monopoles. Now $\cM_B$ canonically embeds into $\cM_H$, as the {\em real} solutions (that is, with vanishing imaginary parts $a=0=\Psi$). Under a certain genericity hypothesis called {\em maximal symmetry breaking} (see \Cref{hypoth:maxsymbr} and \Cref{rem:maxsymbr}), the moduli space $\cM_B$ is a smooth and complete hyperk\"ahler manifold. Its tangent bundle $T \cM_B$ is well defined (as a smooth manifold) and $\cM_B$ also embeds into $T\cM_B$, as the zero section. Under these assumptions, our first main result can be stated as follows:

\begin{Mtheorem}[Existence theorem for Haydys monopoles]\label{Mtheorem:Main_Haydys}
	Under the assumptions of \Cref{hypoth:maxsymbr}:  The Haydys moduli space $\cM_H$ is a smooth manifold around the Bogomolny moduli space $\cM_B$, and the tangent and normal bundles of $\cM_B$ are isomorphic.\\
	In particular, there exists finite energy Haydys monopoles that are not Bogomolny monopoles.
\end{Mtheorem}

\begin{remark}
	This situation contrasts with that of finite energy Kapustin--Witten monopoles on $\mathbb{R}^3$. Indeed, while we find that many Haydys monopoles exist which are not simply Bogomolny monopoles, we proved in \cite{NO21} that any finite energy Kapustin--Witten monopole on $M= \mathbb{R}^3$ must actually be a Bogomolny monopole.
\end{remark}

The moduli space of finite energy solutions to the Haydys monopole equation on $M = \mathbb{R}^3$ inherits some interesting geometric structures mirroring the hyperk\"ahler structure on the moduli space of Bogomolny monopoles $\cM_B$. Namely, we prove that there is an infinite dimensional space 
$$\cC= \cA \times \Omega^0(M, \mathfrak{g}_P) \times \Omega^1(M, \mathfrak{g}_P) \times \Omega^0(M, \mathfrak{g}_P) \cong T (\cA \times \Omega^0(M, \mathfrak{g}_P))$$ 
carrying 3-different hyperk\"ahler structures $(I_1, I_2, I_3)$, $(J_1, J_2, J_3)$, and $(K_1, K_2, K_3)$ associated with the same metric $h$ and such that $I_1=J_1=J_1$, which we denote as $L_1$ for simplicity. The group of gauge transformation $\cG$ acts on $\cC$ preserving all the structure and the space of Haydys monopoles $\cC_H \subset \cC$ can be obtained as the common zero level set of all moment maps. Furthermore $(L_1,h)$ descends to the moduli space of Haydys monopoles $\cM_H$ equipping it with a K\"ahler structure with respect to which $\cM_B \subset \cM_H$ is a minimal Lagrangian submanifold, i.e. a minimal A-brane.\\
Finally, the three different hyperk\"ahler structures on $\cC$ mentioned above, even though not descending to the whole of $\cM_H$, they do descend to $\cM_B$ equipping it with its standard hyperk\"ahler structure.

\begin{Mtheorem}\label{Mtheorem:Geometric_Structures}
	The following assertions hold:
	\begin{enumerate}
		\item[(a)] $\cM_H$ is a K\"ahler manifold with respect to a structure compatible with the $L^2$-metric.
		\item[(b)] $\cM_B \hookrightarrow \cM_H$ as a minimal Lagrangian submanifold.
		\item[(c)] The complex structures $I_2, J_2$, and $K_2$, together with the $L^2$-metric, descend to $\cM_B \hookrightarrow \cM_H$ equipping it with a well defined hyperk\"ahler structure, which is isomorphic to its canonical $L^2$-hyperk\"ahler structure.
	\end{enumerate}
\end{Mtheorem}

\begin{remark}
	In the terminology of \cite{KW07}, part (b) of this theorem is equivalent to saying that $\cM_B \hookrightarrow \cM_H$ is a minimal A-brane.
\end{remark}

\subsection{Organization}  In \Cref{sec:dimred}, we prove that the Haydys monopole \cref{eq:Haydys_Mono_1,eq:Haydys_Mono_2,eq:Haydys_Mono_3} are the dimensional reduction of the 4-dimensional Haydys equation (as in \cite{H15}). In \Cref{sec:Solving_Haydys}, after introducing the necessary tools we prove \Cref{Mtheorem:Main_Haydys} whose proof relies on a use of the Banach space contraction mapping principle. In \Cref{sec:Geometric_Strucutres} we study the geometry of the Haydys monopole moduli space, and prove \Cref{Mtheorem:Geometric_Structures}.

\begin{acknowledgment}
	The authors would like to express their gratitude to Siqi He for having found a mistake in an earlier version of this article.
	The authors would equally like to thank Mark Stern for many helpful conversations and for having taught us so many things about gauge theory. We also thank Paul Aspinwall and Steve Rayan for explaining us what branes are.\\
	We further thank the anonymous referee for the detailed and useful report.\\
	The second author was supported by Funda\c{c}\~ao Serrapilheira 1812-27395, by CNPq grants 428959/2018-0 and 307475/2018-2, and FAPERJ through the program Jovem Cientista do Nosso Estado E-26/202.793/2019.
\end{acknowledgment}

\section{Dimensional reduction}\label{sec:dimred}  In this section we prove the dimensional reductions of the 4-dimensional Haydys equation yield the Haydys monopole \cref{eq:Haydys_Mono_1,eq:Haydys_Mono_2,eq:Haydys_Mono_3}.

Let us start by recalling the notion of complex (anti-)self-duality in dimension 4. Given an oriented, Riemannian 4-manifold $(X, g_4)$, let $\Lambda_{\mathbb{C}}^* X$ be the complexification of its exterior algebra bundle, and let $\ast_4$ and $\overline{\ast}_4$ be the complex linear and conjugate linear extensions of the Hodge star operator on $\Lambda_{\mathbb{C}}^* X$, respectively. Both $\ast_4$ and $\overline{\ast}_4$ square to the identity on $\Lambda_{\mathbb{C}}^2 X$ and hence either can be used to define (anti-)self-dual complex 2-forms. In this paper, we consider the complex linear case, that is when (anti-)self-duality is defined using $\ast_4$.

Let now $\rG$ be a compact Lie group, and $\rG_{\mathbb{C}}$ its complex form. Let $P^X$ be principal $\rG$-bundle over $X$, and define the complexified $\rG_{\mathbb{C}}$-bundle $P_{\mathbb{C}}^X = P^X \times_{\rG} \rG_{\mathbb{C}}$ as being that associated with respect to the standard action by left multiplication of $\rG$ on $\rG_{\mathbb{C}}$. Let $\g_{P^X}$ and $\g_{P_{\mathbb{C}}^X}$ be the corresponding adjoint bundles. Note that $\g_{P_{\mathbb{C}}^X} \simeq \g_{P^X} \otimes_{\mathbb{R}} \mathbb{C}$, and thus
\begin{equation}
	\Lambda_{\mathbb{C}}^k X \otimes_{\mathbb{R}} \g_{P^X} \simeq \Lambda_{\mathbb{C}}^k X \otimes_{\mathbb{C}} \g_{P_{\mathbb{C}}^X} \simeq (\Lambda^k X \otimes \g_{P^X}) \oplus i (\Lambda^k X \otimes \g_{P^X}).
\end{equation}
Any ``complex'' connection $\bA$ on $P_{\mathbb{C}}^X$ decomposes as $\bA = A + i B$, where $A$ is a ``real'' connection on $P^X$ and $B \in \Omega^1 (X, \g_{P^X})$. Then we can decompose the curvature $F_{\bA}$ of $\bA$ as follows
\begin{equation}
	F_{\bA} = \Re(F_{\bA}) + i \ \Im(F_{\bA}).
\end{equation}
and thus
\begin{align}
	\Re (F_{\bA}) &= F_A - \tfrac{1}{2} [B \wedge B],  \\
	\Im (F_{\bA}) &= d_A B.
\end{align}
Let the $\pm$ superscripts denote the pointwise orthogonal projection from $\Lambda^2 X \otimes \g_{P^X}$ onto $\Lambda_\pm^2 X \otimes \g_{P^X}$. Now we can $\bA$ anti-self-dual with respect to $\ast_4$ if
\begin{equation}\label{eq:Haydys_Equation}
	\ast_4 F_{\bA} = - F_{\bA}  \quad  \Leftrightarrow  \quad  \Re(F_{\bA})^+ = 0 = \Im(F_{\bA})^+.
\end{equation}
Note that when $\bA$ is an $\rG$-connection, that is when $B = 0$, then both \cref{eq:Haydys_Equation} reduce to the classical anti-self-duality (instanton) equation on $X$.

Supplementing \cref{eq:Haydys_Equation} with $d_A^* B = 0$ one gets the Haydys equation; cf.  \cite{H15}*{Section~4.2}.

Assume now that $X = \mathbb{S}^1 \times M$, where $M$ is a Riemannian 3-manifold with metric $g$, and $g_X$ is the product metric. Furthermore, let the orientation of $X$ given by the product orientation. The group of orientation preserving isometries of $X$ has a normal subgroup, which is isomorphic to $\SO (2)$, that acts on $\mathbb{S}^1$ as rotations. Thus, one can look for $\SO (2)$-equivariant (``static'') solutions of the Haydys \cref{eq:Haydys_Equation}. It is easy to see, that if $\bA$ is an $\SO (2)$-equivariant connection on $X$, then there exists a principal $\rG$-bundle $P \rightarrow M$, together with and isomorphism between its pullback to $X$ and $P^X$, and a quadruple $(\nabla, \Phi, a, \Psi)$, such that $\nabla$ is a connection on $P$, $a \in \Omega^1 (M, \g_P)$, and $\Phi, \Psi \in \Omega^0 (M, \g_P)$, with the property that (omitting pullbacks and the isomorphism)
\begin{equation}
	\bA = \nabla + \Phi dt + i \ (a + \Psi dt).  \label{eq:invAform}
\end{equation}
Let $\ast$ be the Hodge star operator of $(M, g)$. Then we have the following lemma:

\begin{lemma}[Dimensional reduction of the Haydys equation]
	\label{lem:dimred_H}
	The complex connection $\bA$ in \eqref{eq:invAform} solves the Haydys \cref{eq:Haydys_Equation} and $d_A^* B = 0$, if and only if \cref{eq:Haydys_Mono_1,eq:Haydys_Mono_2,eq:Haydys_Mono_3} hold.
\end{lemma}

\begin{proof} Let $\bA = A + i B$. Recall that the Haydys equations for $\bA$ are
\begin{subequations}
\begin{align}
	\Re(F_{\bA})^+	&= 0,  \label{eq:H1}  \\
	\Im(F_{\bA})^+	&= 0,  \label{eq:H2}  \\
	d_A^{\ast_4} B	&= 0,  \label{eq:H3}
\end{align}
\end{subequations}
where
\begin{equation}
	\Re(F_{\bA}) = F_A - \tfrac{1}{2} [B \wedge B],  \quad  \Im(F_{\bA}) = d_A B.
\end{equation}
Now we further assume that $\bA$ has the form
\begin{equation}
	A = \nabla + \Phi dt,  \quad  B = a + \Psi dt,
\end{equation}
and the quadrupole $(\nabla, \Phi, a, \Psi)$ is pulled back from $M$. Straightforward computations yield
\begin{equation}
	F_A = F_\nabla + d_\nabla \Phi \wedge dt,  \quad  \tfrac{1}{2} [B \wedge B]	= \tfrac{1}{2} [a \wedge a] + [a, \Psi] \wedge dt,
\end{equation}
thus \cref{eq:H1} is equivalent to
\begin{equation}
	\ast F_\nabla - d_\nabla \Phi - \tfrac{1}{2} \ast [a \wedge a] + [a, \Psi] = 0,
\end{equation}
proving \cref{eq:KW_Mono_1,eq:Haydys_Mono_1}.\\
We also have
\begin{equation}
	d_A B = d_\nabla a + \left( d_\nabla \Psi + [a, \Phi] \right) \wedge dt,
\end{equation}
thus for \cref{eq:H2} we have
\begin{equation}
	\Im(F_{\bA})^+ = (d_A B)^+ = \tfrac{1}{2} \left( d_\nabla a - \ast \left( d_\nabla \Psi + [a, \Phi] \right) \right) + \tfrac{1}{2} \left( \ast d_\nabla a - \left( d_\nabla \Psi + [a, \Phi] \right) \right) \wedge dt.
\end{equation}
Thus $\Im(F_{\bA})^+ = 0$ is equivalent to
\begin{equation}
	\ast d_\nabla a - d_\nabla \Psi - [a, \Phi] = 0,
\end{equation}
which proves \cref{eq:Haydys_Mono_2}.\\
Finally, we have
\begin{equation}
	d_A^{\ast_4} B 	= d_{\nabla + \Phi dt}^{\ast_4} (a + \Psi dt) = d_\nabla^* a - [\Phi, \Psi] = d_\nabla^* a + [\Psi, \Phi],
\end{equation}
which proves \cref{eq:Haydys_Mono_3}. This completes the proof.
\end{proof}

Before proceeding, let us point out a couple of other possible ways one can interpret the Haydys monopole equations.

\begin{remark}[Reduction of the Vafa--Witten equations]
	The Vafa--Witten equation is another set of 4-dimensional, gauge theoretic PDE's; cf.  \cite{Haydys2015}*{Section~4.1} for example. We remark, without proof, that the similar reduction of the Vafa--Witten equations also results in \cref{eq:Haydys_Mono_1,eq:Haydys_Mono_2,eq:Haydys_Mono_3}. A simple way to see this is to observe that $\Lambda^1 M \simeq \Lambda_+^2 X$ via the map $b \mapsto \tfrac{1}{2} ( dt \wedge b + \ast b )$.
\end{remark}

\begin{remark}[Reduction of the split $\rG_2$-monopole equation]
	The Haydys monopole \cref{eq:Haydys_Mono_1,eq:Haydys_Mono_2,eq:Haydys_Mono_3} can be obtained as the reduction of the $\rG_2$-monopole equations on $\mathbb{R}^7$ equipped with the split $\rG_2$-structure of signature $(3,4)$. See, for example, \cite{Oliveira2016}*{Section~2}.
\end{remark}

\section{Solving the Haydys monopole equation}\label{sec:Solving_Haydys}

The goal of this section is to construct solutions to the Haydys monopole \cref{eq:Haydys_Mono_1,eq:Haydys_Mono_2,eq:Haydys_Mono_3} on $M = \mathbb{R}^3$ and, more concretely, to prove \Cref{Mtheorem:Main_Haydys}. We achieve that as follows: First, we recall the linearization of the Bogomolny \cref{eq:BPSeq} in \Cref{subsec:Bogomolny_Linearization}, then we introduce the relevant function spaces to be used in \Cref{subsec:Function_Spaces}, and prove a gap theorem for the adjoint of the linearization in \Cref{subsec:Gap_DD}. In \Cref{subsec:Multiplication_Properties}, we prove a multiplication property of the function spaces introduced in \Cref{subsec:Function_Spaces}. In \Cref{subsec:Setup_Haydys_Equation}, we reinterpret the the Haydys monopole \cref{eq:Haydys_Mono_1,eq:Haydys_Mono_2,eq:Haydys_Mono_3} which we supplement in \Cref{subsec:Linearization_Haydys} with the gauge fixing condition. These new set of equations can be viewed as fixed point equation, which we solve using Banach Fixed Point Theorem in \Cref{subsec:Proof_Haydys}. Finally, \Cref{subsec:Dimension_Moduli_Space_Haydys} contains a computation of the dimension of the moduli space, which reveals that our construction yields, in fact, an open subset of the moduli space.

\subsection{The Bogomolny monopole equation and its linearization}\label{subsec:Bogomolny_Linearization}

Let $M = \mathbb{R}^3$ and $m_0 = (\nabla_0, \Phi_0) \in \cA \times \Omega^0 (M, \g_P)$ be a pair satisfying the Bogomolny monopole \cref{eq:BPSeq}, and the finite energy condition, that is $|F_{\nabla_0}| = |d_{\nabla_0} \Phi_0| \in L^2 (\mathbb{R}^3)$. Furthermore, we make the following two hypotheses on $m_0$, which are standard in the literature:
\begin{hypoth}
\label{hypoth:monodecay}
Let $R >0$ and the radius $R$ ball $B_R \subset \mathbb{R}^3$, then $\mathbb{R}^3 \backslash B_R \cong [R, + \infty ) \times \mathbb{S}^2$ and consider the projection $\pi_\infty:\mathbb{R}^3 \backslash B \to \mathbb{S}^2$. We suppose that there is a a unitary connection, $\nabla^\infty$, on a bundle $P_\infty \to \mathbb{S}^2$ which for sufficiently large $R \gg1$ 
\begin{equation}
	P|_{\mathbb{R}^3 \backslash B_R} \cong \pi_\infty^*P_\infty,
\end{equation}
and smooth nonzero sections $\Phi_\infty, \kappa \in \Omega^0 (\mathbb{S}^2, \ad (P_\infty))$, such that $\Phi_0$ and $d_{\nabla_0} \Phi_0$ have the asymptotic expansions
\begin{subequations}
\begin{align}
	\Phi_0 				&= \pi_\infty^* \Phi_\infty - \tfrac{1}{2 r} \pi_\infty^* \kappa + O \left( r^{-2} \right),  \label{eq:limPhi}  \\
	d_{\nabla_0} \Phi_0 &= \tfrac{1}{2 r^2} \pi_\infty^* \kappa \otimes dr + O \left( r^{-3} \right) .  \label{eq:limDPhi}
\end{align}
\end{subequations}
Furthermore, we have
\begin{equation}
\nabla^\infty \Phi_\infty = 0, \quad F_{\nabla^\infty} = \tfrac{1}{2} \kappa \otimes \vol_{\mathbb{S}_\infty^2}, \quad [\Phi_\infty, \kappa] = 0.
\end{equation}
\end{hypoth}

\begin{remark}\label{rem:monodecay}
	\Cref{hypoth:monodecay} has been proven in some cases (see \cite{JT80}*{Theorem~10.5} for the first example) and, conjecturally, holds for all finite energy monopoles on $\mathbb{R}^3$; cf.  \cite{JT80}*{Theorem~18.4~\&~Corollary~18.5}. The authors of this paper, together with Benoit Charbonneau, are currently working on a proof of this conjecture.
\end{remark}

We furthermore impose a more technical hypothesis, which is crucial in the proof of \Cref{Mtheorem:Main_Haydys}.
\begin{hypoth}\label{hypoth:maxsymbr}
	The field $\Phi_\infty$ takes values in the interior of a Weyl chamber, i.e. it attains no value in a facet of one such Weyl chamber. A Bogomolny monopole satisfying this hypothesis is said to have {\em maximal symmetry breaking}.
\end{hypoth}

\begin{remark}\label{rem:maxsymbr}
	It is easy to see that a monopole has maximal symmetry breaking exactly if $\ker (\ad (\Phi_\infty))$ is Abelian. Note also that any nonflat monopole with structure group $\rG = \SU(2)$ has to have maximal symmetry breaking. This said, we point out that monopoles without maximal symmetry breaking do exist; see \cites{D92,D93}.\\
	Furthermore, we mention here, without proof, that by adapting the arguments in \cite{CLS16}, it is possible to prove that maximal symmetry breaking implies \Cref{hypoth:monodecay}. The proof of this---among more general claims---is currently being completed by the authors; see \Cref{rem:monodecay}.\\
	Finally, we conjecture that \Cref{Mtheorem:Main_Haydys} holds for monopoles with nonmaximal symmetry breaking as well.
\end{remark}

\begin{definition}\label{def:tangent}
	For any $v_0 = (a_0, \Psi_0) \in \Omega^1 (M, \g_P) \oplus \Omega^0 (M, \g_P)$, let
	\begin{align}
		d_2 (v_0) 	&= \ast d_{\nabla_0} a_0 - d_{\nabla_0} \Psi_0 - [a_0, \Phi_0 ],  \\
		d_1^* (v_0) &= d_{\nabla_0}^* a_0 -[\Phi_0, \Psi_0],
	\end{align}
	and $D = d_2 \oplus d_1^*$, let $D^*$ be the formal adjoint of $D$. A pair $v_0 = (a_0, \Psi_0) \in \Omega^1 (M, \g_P) \oplus \Omega^0 (M, \g_P)$ is called a {\em tangent vector to the Bogomolny monopole moduli space at $m_0$} if it lies in $\ker_{L_1^2} (D)$.\\
	Moreover, for any $c_0$, if we define
	\begin{equation}
		(d_{\nabla_0} \Phi_0)^W (c_0) = \left( \ast [a_0 \wedge d_{\nabla_0} \Phi_0] - [d_{\nabla_0} \Phi_0, \Psi_0], [\langle d_{\nabla_0} \Phi_0,  a_0 \rangle] \right) \in \Omega^1 (M, \g_P) \oplus \Omega^0 (M, \g_P).
	\end{equation}
\end{definition}

It is easy to see that
\begin{align}
	DD^* c &= \nabla_0^* \nabla_0 c - [[c, \Phi_0], \Phi_0],  \\ 
	D^*D c &= DD^* c + 2 (d_{\nabla_0} \Phi_0)^W(c).
\end{align}

\begin{remark}
	The operator $D$ is the linearization of the Bogomolny monopole \cref{eq:BPSeq} together with the standard Coulomb type gauge fixing condition $d_{\nabla_0}^* a_0 = [\Phi_0, \Psi_0]$. Thus, when $D : L_1^2 \rightarrow L^2$ is surjective the implicit function theorem can be used to prove that the Bogomolny monopole moduli space is smooth and its tangent space at $m_0$ can be identified with $\ker_{L_1^2} (D)$.
\end{remark}

\subsection{Function spaces}\label{subsec:Function_Spaces}

Now we introduce the various functions spaces that are used in the proof of \Cref{Mtheorem:Main_Haydys}.

\begin{definition}  Let $\| \cdot \|$ denote $L^2$-norms and $\rho = (1+|x|^2)^{1/2}$. Define the Hilbert spaces $\cH_k$ (with $k = 1,2$, or 3) as the norm-completions of $C_0^\infty (\mathbb{R}^3,( \Lambda^0 \oplus \Lambda^1) \otimes \g_P)$ via the norms
\begin{align}
	\| c \|_{\cH_0}^2 &= \| c \|^2,  \\ 
	\| c \|_{\cH_1}^2 &= \| \nabla_0 c \|^2 + \| \rho^{-1}  c \|^2 + \| [\Phi_0, c] \|^2,  \\
	\| c \|_{\cH_2}^2 &= \| \nabla_0^2 c \|^2 + \| \rho^{-1} \nabla_0 c \|^2 + \| [\Phi_0, \nabla_0 c ] \|^2  +  \| [\Phi_0 , [\Phi_0, c] ] \|^2+ \| [\Phi_0, \rho^{-1} c ] \|^2 + \| \rho^{-2}  c \|^2 .
\end{align}
The corresponding inner products are denoted by $\langle \cdot, \cdot \rangle_{\cH_k}$.
\end{definition}

\begin{lemma}[Sobolev and Hardy's inequalities]\label{lem:Sobolev_Hardy}
	Let $c \in \cH_1$. Then the Sobolev inequality says
	\begin{equation}
		\| \nabla_0 c \| \geqslant \tfrac{3}{2} \| c \|_{L^6 (\mathbb{R}^3)},
	\end{equation}
	and Hardy's inequality says
	\begin{equation}
		\| \nabla_0 c \| \geqslant \tfrac{1}{2} \| \rho^{-1} c \|_{L^2 (\mathbb{R}^3)}.
	\end{equation}
\end{lemma}

\begin{remark}
	Let the symbol $\sim$ denote norm equivalence. Then using Hardy's inequality and \Cref{hypoth:monodecay} gives
	\begin{align}
		\| c \|_{\cH_1}^2 &\sim \| \nabla_0 c \|^2 + \| [\Phi_0, c] \|^2,  \\
		\| c \|_{\cH_2}^2 &\sim \| \nabla_0^2 c \|^2 + \| [\Phi_0, \nabla_0 c ] \|^2 + \| \rho^{-2} c \|^2.
	\end{align}
\end{remark}

\begin{lemma}[First order inequalities, partly in \cite{Taubes1983}*{Lemma~6.8}]\label{lem:Inequality_D}
	There are positive constants $R,C_{m_0}$, depending only on the monopole $m_0$, such that for all $c_0 \in \cH_1$
	\begin{align}
		\| c_0 \|_{\cH_1}^2 & \leqslant C_{m_0} \|D^*c_0 \|_{\cH_0}^2,  \\
		\| c_0 \|_{\cH_1}^2 & \leqslant C_{m_0} \left( \|Dc_0 \|_{\cH_0}^2 + \| c_0 \|_{L^2(B_{R}(0))}^2 \right),  \\
		\| c_0 \|_{\cH_2}^2 & \leqslant C_{m_0} \left( \|D^*c_0 \|_{\cH_1}^2 + \|c_0 \|_{L^2(B_{R}(0))}^2 \right),  \\
		\| c_0 \|_{\cH_2}^2 & \leqslant C_{m_0} \left( \|Dc_0 \|_{\cH_1}^2 + \|c_0 \|_{L^{2}(B_{R}(0))}^2 \right).
	\end{align}
\end{lemma}
\begin{proof}
	We start by proving the first inequality, which follows from the Weitzenb\"ock type formula $DD^* c = \nabla_0^* \nabla_0 c - [[c, \Phi_0], \Phi_0] $ as follows
	\begin{equation}
		\| D^* c \|^2 = \langle c , DD^* c \rangle = \| \nabla_0 c \|^2 + \| [\Phi_0, c] \|^2 \sim \| c \|_{\cH_1}^2.
	\end{equation}
	The method for proving the second inequality is outlined in \cite{Taubes1983}. For completeness, we include here its proof using the strategy outlined in that reference. By \Cref{hypoth:monodecay}, we have that $\rho^2 d_{\nabla_0} \Phi_0 \in L^\infty$, and hence, if $c \in \cH_1$, then we can integrate by parts and thus
	\begin{equation}\label{eq:Dc_NormSquared}
	\begin{aligned}
		\| D c \|_{\cH_0}^2 &= \langle c, D^*D c \rangle = \langle c, \nabla_0^* \nabla_0 c - [[c, \Phi_0], \Phi_0] + 2(d_{\nabla_0} \Phi_0)^W(c) \rangle  \\ 
							&= \| \nabla_0 c \|^2 + \| [\Phi_0, c] \|^2 + 2 \langle c, (d_{\nabla_0} \Phi_0)^W c \rangle  \\
							&= \| c \|_{\cH_1}^2 + 2 \langle c, (d_{\nabla_0} \Phi_0)^W c \rangle.
	\end{aligned}
	\end{equation}
	Now, for $R \gg 1$ let $\chi_R$ be a smooth bump function supported in $B_{R}(0)$ and equal to $1$ in $B_{R-1}(0)$. Then,
	\begin{equation}
		\langle c, (d_{\nabla_0} \Phi_0)^W c \rangle =  \langle \chi_R c, (d_{\nabla_0} \Phi_0)^W c \rangle +  \langle (1-\chi_R)c, (d_{\nabla_0} \Phi_0)^W c \rangle,
	\end{equation}
	with the first terms satisfying 
	\begin{equation}
		| \langle \chi_R c, (d_{\nabla_0} \Phi_0)^W c \rangle | \leqslant \left( \sup_{x \in B_{R}(0) } |d_{\nabla_0} \Phi_0(x)| \right) \ \| c\|_{L^2(B_R(0))}^2.
	\end{equation}
	As for the second term, we may use the particular form of $(d_{\nabla_0} \Phi_0)^W c$ and the $\Ad$-invariance of the inner product to find a bilinear map $N(\cdot , \cdot)$ so that
	\begin{equation}
		| \langle (1-\chi_R)c, (d_{\nabla_0} \Phi_0)^W c \rangle| \lesssim \|(1-\chi_R)d_{\nabla_0} \Phi_0 \|_{L^2} \| N(c,c) \|_{L^2} \lesssim \|(1-\chi_R)d_{\nabla_0} \Phi_0 \|_{L^2} \| c \|_{\cH_1}^2,
	\end{equation}
	where in the last inequality we have used \Cref{lem:Multiplication_H_Spaces} to be proven later. Now, given that $d_{\nabla_0} \Phi_0 \in L^2$, for any positive $\epsilon \ll 1$ we may find $R=R_{\epsilon} \gg 1$ so that $\|(1-\chi_{R_{\epsilon}})d_{\nabla_0} \Phi_0 \|_{L^2} \leqslant \epsilon$ and so
	\begin{equation}
		| \langle (1-\chi_{R_{\epsilon}})c, (d_{\nabla_0} \Phi_0)^W c \rangle| \lesssim \epsilon \| c \|_{\cH_1}^2.
	\end{equation}
	Inserting these back into \cref{eq:Dc_NormSquared} we find
	\begin{equation}
		\| D c \|_{\cH_0}^2 \gtrsim (1-\epsilon) \| c \|_{\cH_1}^2 - \left( \sup_{x \in B_{R_{\epsilon} }(0) } |d_{\nabla_0} \Phi_0(x)| \right) \ \| c\|_{L^2(B_{R_{\epsilon}} (0))}^2,  \label[ineq]{ineq:Inequality_D_Intermediate}
	\end{equation}
	and rearranging yields the second inequality in the statement.

	We now turn to proving the last two inequalities, i.e. the last two ones. Using the first inequality just proved above we compute
	\begin{align}
		\| c \|_{\cH_2}^2 & \lesssim \| \nabla_0 c \|_{\cH_1}^2 + \| [ \Phi_0 , c] \|_{\cH_1}^2 + \| \rho^{-1}c \|_{\cH_1}^2  \\
		&\lesssim \| D^*\nabla_0 c \|_{\cH_0}^2 + \| D^* [ \Phi_0 , c] \|_{\cH_0}^2 + \| D^* \rho^{-1}c \|_{\cH_0}^2  \\
		&\lesssim \|\nabla_0 D^* c - B (\nabla_0 \Phi_0, c) \|_{\cH_0}^2 + \| [ \Phi_0, D^*c] + B (\nabla_0 \Phi_0, c) + [\Phi_0, [\Phi_0, c]] \|_{\cH_0}^2  \\
		&\quad + \| \rho^{-2} c + \rho^{-1} D^* c \|_{\cH_0}^2  \\
		&\lesssim \|\nabla_0 D^* c \|_{\cH_0}^2 + \| [ \Phi_0 , D^*c] \|_{\cH_0}^2 + \| [\Phi_0 , [\Phi_0 , c]] \|_{\cH_0}^2 + \| \rho^{-1}D^*c \|_{\cH_0}^2 + \| B (\nabla_0 \Phi_0, c) \|_{\cH_0}^2  \\
		&\quad + \| \rho^{-2} c \|_{\cH_0}^2  \\ \label[ineq]{ineq:Inequality_Dc_H1_Intermediate}
		&\lesssim \| D^* c \|_{\cH_1}^2  + \| [\Phi_0 , [ \Phi_0 , c ] ] \|_{\cH_0}^2 + \| \rho^{-2} c \|_{\cH_0}^2 ,
	\end{align}
	where $B (-, -)$ denotes a bilinear operator which is algebraic, and thus continuous. Now, using the definition of the $\cH_1$-norm we have
	\begin{equation}\label{eq:Norm_H1_D^*c_Intermediate}
		\| D^* c \|_{\cH_1}^2 = \| \nabla_0 D^*c \|^2 + \| [\Phi_0, D^* c] \|^2 ,
	\end{equation}
	and we consider each of these separately. For the first term we use the Hardy's inequality together with Young's inequality in the form $2 \langle D^* (\rho^{-1}  c) , \rho^{-2} d \rho \otimes c\rangle \leqslant \sqrt{2} \| D^* (\rho^{-1}  c) \| + \tfrac{1}{\sqrt{2}} \| \rho^{-2}  c \|^2$
	\begin{align}
		\| \nabla_0 D^*c \|^2 & \gtrsim \| \rho^{-1} D^* c \|^2 \gtrsim \| D^* (\rho^{-1}  c) - \rho^{-2} d \rho \otimes c \|^2 \\
		& \gtrsim (1- \sqrt{2})\| D^* (\rho^{-1}  c) \|^2 + (1-\tfrac{1}{ \sqrt{2}})\| \rho^{-2}  \otimes c \|^2 \\	
		& \gtrsim \frac{3 - \delta - 2 \sqrt{2}}{2}  \|\rho^{-2}  c \|^2  + \delta \| \nabla (\rho^{-1}  c)  \|^2 + \delta \| [ \Phi_0 , \rho^{-1}  c ] \|^2 \\
		& \gtrsim  \|\rho^{-2}  c \|^2  + \| \nabla (\rho^{-1}  c)  \|^2 + \| [ \Phi_0 , \rho^{-1}  c ] \|^2 ,	
	\end{align}
	for some fixed $\delta \in (0,3-2\sqrt{2})$.
	As for the second term, we use \cref{hypoth:monodecay}, namely that $d_{\nabla_0} \Phi_0 =O(\rho^{-2})$, and an argument as that made above to control $\langle ( d_{\nabla_0} \Phi_0 )^{W} c, c \rangle$
	\begin{align}
		\| [\Phi_0, D^*c] \|^2 & \gtrsim \| D^* [ \Phi_0 , c ] - (d_{\nabla_0} \Phi_0  )^{W} c \|^2	 \\
		& \gtrsim \| D^* [ \Phi_0 , c ] \|^2  - \|( d_{\nabla_0} \Phi_0  )^{W} c \|^2 \\
		& \gtrsim \| [\Phi_0 , [ \Phi_0 , c ] ] \|^2  - \langle ( d_{\nabla_0} \Phi_0 )^{W} \rho^{-1}c, \rho^{-1} c \rangle \\
		& \gtrsim \| [\Phi_0 , [ \Phi_0 , c ] ] \|^2 - \| c \|_{L^2(B_{R_{\epsilon}})}^2 - \epsilon \| \rho^{-1} c \|^2_{\cH_1} . 
	\end{align}
	We now sum these inequalities, i.e. insert them back into \cref{eq:Norm_H1_D^*c_Intermediate}, and recall that $ \| \nabla_0( \rho^{-1} c ) \|^2 + \| [\Phi_0, \rho^{-1} c]  \|^2 = \| \rho^{-1} c \|^2_{\cH_1} $ to obtain
	\begin{align}
		\| D^* c \|_{\cH_1}^2 & \gtrsim \|\rho^{-2}  c \|^2  + \| [\Phi_0 , [ \Phi_0 , c ] ] \|^2 +\| \rho^{-1} c \|^2_{\cH_1} - \| c \|_{L^2(B_{R_{\epsilon}})}^2 - \epsilon \| \rho^{-1} c \|^2_{\cH_1} \\
		&\gtrsim \|\rho^{-2}  c \|^2  + \| [\Phi_0 , [ \Phi_0 , c ] ] \|^2 +\| \rho^{-1} c \|^2_{\cH_1} - \| c \|_{L^2(B_{R_{\epsilon}})}^2 ,
	\end{align}
	where we have chosen $\epsilon>0$ sufficiently small so it may be absorbed. Then, rearranging we obtain
	\begin{equation}
		\|\rho^{-2}  c \|^2 + \| [\Phi_0 , [ \Phi_0 , c ] ] \|^2 +\| \rho^{-1} c \|^2_{\cH_1} \lesssim  \| D^* c \|_{\cH_1}^2 + \| c \|_{L^2(B_{R})}^2,
	\end{equation}
	for some $R>0$. Then, inserting this into the \cref{ineq:Inequality_Dc_H1_Intermediate} we obtain
	\begin{equation}
		\| c \|^2_{\cH_2} \lesssim \| D^* c \|_{\cH_1}^2 + \| c \|_{L^2(B_{R})}^2,
	\end{equation}
	which proves the third inequality in the statement. The proof of the last one follows from a very similar computation, which we omit.
\end{proof}

\begin{lemma}[Second order inequalities]
	There is $C>0$ depending only on the monopole $m_0$ so that for all $c_0 \in \cH_2$
	\begin{subequations}
	\begin{align}
		\| c_0 \|_{\cH_2}^2 & \leqslant C \ \left( \|DD^*c_0 \|_{\cH_0}^2 +  \| c_0 \|_{L_1^{2} (B_{R_{\epsilon}}(0))}^2 \right),  \label[ineq]{ineq:Inequalities_Second_Order_1}  \\
		\| c_0 \|_{\cH_2}^2 & \leqslant C \ \left(  \|D^*Dc_0 \|_{\cH_0}^2 + \| c_0 \|_{L^2 (B_{R_{\epsilon}}(0))}^2 \right).  \label[ineq]{ineq:Inequalities_Second_Order_2}
	\end{align}
	\end{subequations}
\end{lemma}

\begin{proof}
	The last two inequalities in the statement of \Cref{lem:Inequality_D} yield 
	\begin{align}
		\| c_0 \|_{\cH_2}^2 & \lesssim  \|D^*c_0 \|_{\cH_1}^2 + \| c_0 \|_{L^2(B_{R}(0))}^2 \\
		\| c_0 \|_{\cH_2}^2 & \lesssim  \|Dc_0 \|_{\cH_1}^2 + \| c_0 \|_{L^2(B_{R}(0))}^2 .
	\end{align}
	Composing these with the first two in \Cref{lem:Inequality_D} yields
	\begin{align}
		\| c_0 \|_{\cH_2}^2 & \lesssim  \|DD^*c_0 \|_{\cH_0}^2 +  \| D^*c_0 \|_{L^2(B_{R}(0))}^2 + \| c_0 \|_{L^2(B_{R}(0))}^2 \\ 
		\| c_0 \|_{\cH_2}^2 & \lesssim   \|D^*Dc_0 \|_{\cH_0}^2 + \| c_0 \|_{L^2(B_{R}(0))}^2 ,
	\end{align}  
	which yields the stated inequalities.
\end{proof}

\begin{corollary}\label{cor:Fredholm}
	For $k = 1,2$, the operators $D, \ D^* : \cH_{k + 1} \rightarrow \cH_k$ are continuous and Fredholm. In particular $DD^*, \ D^*D : \cH_2 \rightarrow \cH_0$ are also continuous and Fredholm.\footnote{When $D$ and $D^*$ are considered as maps from $L_1^2 (\mathbb{R}^3)$ to $L^2 (\mathbb{R}^3)$, \Cref{cor:Fredholm} does not hold, because of the failure of the Rellich Lemma for unbounded domains.}
\end{corollary}

\begin{proof}
	The continuity of these operators is immediate so we focus on the Fredholm property. The inequalities in \Cref{lem:Inequality_D} together with the compactness of the embedding $\cH_1 \hookrightarrow L^2(B_{R_{\epsilon}})$, imply that the operators $D,D^* : \cH_1 \rightarrow \cH_0$ have finite dimensional kernel and closed image. Similarly, these inequalities show that both $\ker_{L^2} (D^*)$ and $\ker_{L^2}(D)$ are contained in $\cH_1$. Hence, these can be respectively identified with the cokernel of the operators $D, D^* : \cH_1 \rightarrow \cH_0$, and so their cokernels are also finite dimensional. Putting all these facts together follows that the mentioned first order operators are Fredholm.\\
	In order to prove that the second order operators $DD^*, \ D^*D : \cH_2 \rightarrow \cH_0$ are Fredholm is enough that $D$ and $D^*$ also be Fredholm when defined from $\cH_2$ to $\cH_1$, which can be done through a very similar computation. Alternatively, it follows from the same argument as above, but using \cref{ineq:Inequalities_Second_Order_1,ineq:Inequalities_Second_Order_2} instead.
\end{proof}

\subsection{A Gap Theorem}\label{subsec:Gap_DD}

In the proof of \Cref{lem:Inequality_D} we saw that  
\begin{equation}
	\| D^* c \|^2 = \| \nabla_0 c \|^2 + \| [\Phi_0, c] \|^2,
\end{equation}
which implies the operator $D^* : \cH_1 \rightarrow \cH_0$ is injective. Indeed, from \Cref{lem:Inequality_D}, the inequality
\begin{equation}\label[ineq]{ineq:D^*_Gap}
	\| D^* c \|^2_{\cH_0} \gtrsim  \|  c \|^2_{\cH_1} ,
\end{equation}
holds for any $c \in \cH_1$. Thus, $DD^* : \cH_2 \rightarrow \cH_0$ is also injective, and, since it is a formally self-adjoint elliptic operator, its spectrum is gapped.

\begin{theorem}[Gap Theorem]
There is a constant $C > 0$, possibly depending on the monopole $m_0 = (\nabla_0, \Phi_0)$, such that
\begin{equation}
	\| DD^* c \|^2_{\cH_0} \geqslant C \|  c \|^2_{\cH_2}.  \label[ineq]{ineq:gap}
\end{equation}
\end{theorem}

\begin{proof}
The proof of this assertion is a standard argument by contradiction using the fact that $\ker_{\cH_2} (DD^*) \subseteq \ker_{\cH_2} (D^*) = \{ 0 \}$. Indeed, if \cref{ineq:gap} is not true, then there is a sequence $c_i \in \cH_2$ such that $\| c_i \|_{\cH_2}=1$ and $ \| DD^*c_i \|_{\cH_0}^2 \rightarrow 0$. Since the sequence $\{ c_i \}$ is bounded in $\cH_2$ by assumption, there is a weak $\cH_2$-limit, say $c_\infty \in \cH_2$, which satisfies $DD^* c_\infty = 0$. As $DD^*$ has no $\cH_2$-kernel, we have $c_\infty = 0$, that is $c_i$ converges weakly to 0.

Now, consider any bounded domain $\Omega \subset \mathbb{R}^3$, the embedding $\cH_2 (\Omega) \hookrightarrow L_1^2 (\Omega)$ is compact, $c_i \rightarrow c_\infty=0$ strongly in $L_1^2 (\Omega)$. Putting this together with \cref{ineq:Inequalities_Second_Order_1} for $DD^*$ follows that 
\begin{equation}
	\| c_i \|_{\cH_2}^2  \leqslant C \ \left( \|DD^*c_i \|_{\cH_0}^2 +  \| c_i \|_{L_1^2 (B_R(0))}^2 \right) .
\end{equation}  
Taking the limit as $i \rightarrow \infty$, the right hand side converges to zero, so we also have
\begin{equation}
	\lim\limits_{i \rightarrow \infty} \|c_i \|_{\cH_2} = 0,
\end{equation}
which contradicts $\| c_i \|_{\cH_2} = 1$.
\end{proof}

Using the Lax--Milgram Theorem, we immediately conclude the following:

\begin{corollary}[Green operator of $DD^*$]\label{cor:Green}
	There is continuous linear map 
	\begin{equation}
		G: \cH_0 \rightarrow \cH_2,
	\end{equation}
	such that $G \circ DD^* = \id_{\cH_2}$.
\end{corollary}

\subsection{Multiplication properties of the function spaces}\label{subsec:Multiplication_Properties}

\begin{lemma}\label{lem:Multiplication_L2}
	Let 
	\begin{equation}
		N_0 : ( \Lambda^1 \oplus \Lambda^0 ) \oplus ( \Lambda^1 \oplus \Lambda^0 ) \rightarrow \mathbb{R},
	\end{equation}
	be a bilinear map whose norm is pointwise uniformly bounded, that is there is a positive constant $C$, such that for all $x \in \mathbb{R}^3$, and all $\gamma_1, \gamma_2 \in \Lambda_x^1 \oplus \Lambda_x^0$
	\begin{equation}
		| N_0 (\gamma_1, \gamma_2) | \leqslant C |\gamma_1| |\gamma_2|.
	\end{equation}
	Then, for any connection $\nabla$, there is some other constant $C'$ such that if now $c_1, c_2$ are $L_1^2$ sections, then
	\begin{equation}
		\| N_0 (\gamma_1, \gamma_2) \| \leqslant C' \| \nabla \gamma_1 \| (\| \gamma_2 \| + \| \nabla \gamma_2 \|).
	\end{equation}
\end{lemma}

\begin{proof}
	Given that $\tfrac{3}{2} < 2 < 3$ and using the H\"older's inequality twice and then the Sobolev inequality from \Cref{lem:Sobolev_Hardy}, we have
	\begin{align}
		\| N_0 (\gamma_1, \gamma_2) \| &\lesssim \| \ |\gamma_1| \ |\gamma_2| \ \|  \\
		&\lesssim \| \ |\gamma_1| \ |\gamma_2| \ \|_{L^3} + \| \ |\gamma_1| \ |\gamma_2| \ \|_{L^{3/2}} \\
		&\lesssim \| \gamma_1 \|_{L^6} \| \gamma_2 \|_{L^6} + \| \gamma_1 \|_{L^6} \| \gamma_2 \|_{L^2}  \\
		&\lesssim \| \nabla \gamma_1 \| \ \| \nabla \gamma_2 \| + \| \nabla \gamma_1 \| \ \| \gamma_2 \|,
	\end{align}
	which concludes the proof.
\end{proof}

The nonlinearities of the Haydys monopole \cref{eq:Haydys_Mono_1,eq:Haydys_Mono_2,eq:Haydys_Mono_3} are quadratic, but come composed with the Lie algebra bracket $[\cdot, \cdot]$ acting in the $\g$-valued components. Thus, given a quadratic map $N_0$ as in \Cref{lem:Multiplication_L2}, the maps under consideration are of the form $N (c_1, c_2)$, that is if for $i=1,2$: $c_i = s_i \otimes \gamma_i$ with $s_i \in \g$, $\gamma_i \in \Lambda^0 \oplus \Lambda^1$ we have 
\begin{equation}
	N (c_1, c_2) = [s_1, s_2] \otimes N_0 (\gamma_1, \gamma_2).
\end{equation}
In that context, and in terms of the $\cH_k$-norms, the result in the \Cref{lem:Multiplication_L2} must be rephrased, which requires some preparation. For instance, recall that given a finite energy monopole $m_0=(\nabla_0,\Phi_0)$, there is $\Phi_\infty: \mathbb{S}^2 \rightarrow \g - \lbrace 0 \rbrace$ so that $\lim_{\rho \rightarrow \infty} \Phi_0 |_{\mathbb{S}^2_\rho} = \Phi_\infty$ uniformly. Recall that a monopole $m_0$ is said to have maximal symmetry breaking, if $\ker(\ad (\Phi_\infty))$ is Abelian.

\begin{lemma}[Multiplication properties of the $\cH_k$-spaces]\label{lem:Multiplication_H_Spaces}
	Let $m_0 = (\nabla_0,m_0)$ be a monopole with maximal symmetry breaking and $N$ a quadratic map as above. Then there exists $C_{m_0} > 0$, possibly depending on the monopole $m_0 = (\nabla_0,\Phi_0)$, such that
	\begin{equation}
		\| N (c_1, c_2) \|_{\cH_0} \leqslant C_{m_0} \|c_1 \|_{\cH_1} \|c_2 \|_{\cH_1}.
	\end{equation}
\end{lemma}
\begin{proof}
	We start by proving the claimed inequality inside a bounded domain $\Omega \subset \mathbb{R}^3$. Notice that, as $\Omega$ is bounded and $\rho$ continuous the $H^1(\Omega)$-norm is equivalent to $L_1^2 (\Omega)$ norm. Then, from \Cref{lem:Multiplication_L2} we immediately obtain 
	\begin{equation}
		\| N(c_1,c_2) \|_{\cH_0(\Omega)} \leqslant C_{m_0} \|c_1 \|_{\cH_1(\Omega)} \|c_2 \|_{\cH_1(\Omega)}.
	\end{equation}
	We are then left with proving the inequality in the statement outside a compact domain. By \Cref{hypoth:monodecay}, the Higgs field $\Phi_0$ approaches a $\nabla^\infty$-parallel field $\Phi_\infty \neq 0$ at infinite. Hence, there exists an $R > 0$, such that $|\Phi_0| > |\Phi_\infty|/2$. On the other hand, \Cref{hypoth:monodecay,hypoth:maxsymbr} implies that $\ker(\ad (\Phi_0) )$ is Abelian on $\mathbb{R}^3 - B_R (0)$. Let $\underline{\g}$ denote the trivial $\Lie (G)$-bundle, which we identify now with $\ad (P)$. Furthermore, let $\underline{\g}^{||} = \ker(\ad (\Phi_0))$. Then, over $\mathbb{R}^3 - B_R (0)$, $\underline{\g}^{||}$ is a smooth, Abelian Lie-algebra bundle, and
	\begin{equation}
		\underline{\g} \simeq \underline{\g}^{||} \oplus \underline{\g}^{\perp},
	\end{equation}
	with $\underline{\g}^{\perp} = \ker(\ad (\Phi_0))^{\perp}$. Sections of $(\Lambda^0 \oplus \Lambda^2) \otimes \underline{\g}$ can be then written as $c = c^{||} + c^{\perp}$. By the maximal symmetry breaking hypothesis again, we have that $[\underline{\g}^{||}, \underline{\g}^{||}] = \{ 0 \}$. Thus 
	\begin{equation}
		[c_1, c_2] =  [c_1^{||}, c_2^{\perp}] + [c_1^{\perp}, c_2^{||}] + [c_1^{\perp}, c_2^{\perp}].
	\end{equation}
	Thus, it suffices to prove the stated inequality separately for each of these components. Start by noticing that $\|c^{\perp}\|_{\cH_1} \sim \|c^{\perp}\|_{L_1^2 }$ while $\|c^{||}\|_{\cH_1} \sim \|\nabla_0 c^{||}\|_{L^2} + \| \rho^{-1}c^{\perp}\|_{L^2} \sim \|\nabla_0 c^{||}\|_{L^2} $, and so making use of \Cref{lem:Multiplication_L2} we have
	\begin{equation}
		\| [c_1^{\perp}, c_2^{||} ] \| \lesssim (\| c_1^{\perp} \| + \| \nabla c_1^{\perp} \|) \| \nabla c_2^{||} \| \lesssim \| c_1^{\perp} \|_{\cH_1} \| c_2^{||} \|_{\cH_1} \lesssim \| c_1 \|_{\cH_1} \| c_2 \|_{\cH_1}.
	\end{equation}
	A similar application of \Cref{lem:Multiplication_L2} with the roles of $c_1$, $c_2$ interchanged gives the same bound on the term $[c_1^{\perp}, c_2^{||}]$, and in any order regarding $c_1$, $c_2$ also one for the term $[c_1^{\perp}, c_2^{\perp}]$.
\end{proof}

\subsection{Preparation for the proof of the \Cref{Mtheorem:Main_Haydys}}\label{subsec:Setup_Haydys_Equation}

Let $\cD = \cA \oplus \Omega^0 (M, \g_P) $ and $\cC = T \cD = \cA \oplus \Omega^0 (M, \g_P) \oplus \Omega^1 (M, \g_P) \oplus \Omega^0 (M, \g_P)$ be the configuration space for Haydys monopoles, that is $(\nabla,\Phi, a, \Psi) \in \cC$. Furthermore, let $\cR = \Omega^1 (M, \g_P) \oplus \Omega^1 (M, \g_P) \oplus \Omega^0 (M, \g_P)$, and
\begin{equation}
	\kappa : \cC \rightarrow \cR; \ (\nabla,\Phi, a, \Psi)  \mapsto  \begin{pmatrix}  \ast F_\nabla - d_\nabla \Phi - \tfrac{1}{2} \ast [a \wedge a] + [a, \Psi]  \\  \ast d_\nabla a - d_\nabla \Psi - [a, \Phi]  \\  d_\nabla^* a + [\Psi, \Phi]  \end{pmatrix}.
\end{equation}
Then we can rewrite the Haydys monopole \cref{eq:Haydys_Mono_1,eq:Haydys_Mono_2,eq:Haydys_Mono_3} as
\begin{equation}
	\kappa (\nabla,\Phi, a, \Psi) = 0.
\end{equation} 
We call $\kappa$ the {\em Haydys map}.

\begin{remark}\label{rem:Almost_KWm}
Note that if $m_0 = (\nabla_0, \Phi_0)$ is a Bogomolny monopole on $M$, then for $v_0 = (a_0, \Psi_0) \in \Omega^1 (M, \g_P) \oplus \Omega^0 (M, \g_P)$ the last two components of $\kappa (m_0,v_0)$ are exactly the linearization of the Bogomolny \cref{eq:BPSeq} together with the usual Coulomb type gauge fixing condition which we have been writing as $D (v_0) = (d_2 \oplus d_1^*) (v_0)$; see \Cref{subsec:Bogomolny_Linearization}. Let now $v_0$ be a tangent vector to the Bogomolny monopole moduli space at $m_0$ as in \Cref{def:tangent}. Then the last two components and the terms involving $m_0$ in the first component of $\kappa (m_0, v_0)$ vanish, but the (quadratic) terms depending on $v_0$ need not, in general. So $(m_0, \nu_0)$ fails to solve the Haydys monopole \cref{eq:Haydys_Mono_1,eq:Haydys_Mono_2,eq:Haydys_Mono_3} for $v_0 \in \ker(D) - \lbrace 0 \rbrace$, but the error is of order $O(|v_0|^2)$ pointwise.
\end{remark}

\subsection{Linearization and gauge fixing}\label{subsec:Linearization_Haydys}

In this section, we look for solutions of the Haydys monopole \cref{eq:Haydys_Mono_1,eq:Haydys_Mono_2,eq:Haydys_Mono_3} as follows:  Let $m_0$ be a (finite energy) Bogomolny monopole, and $v_0 = (a_0, \Psi_0)$ be a tangent vector to the Bogomolny monopole moduli space at $m_0$ with unit $L^2$-norm. Then consider $c_0 = (m_0, t v_0)=(\nabla_0, \Phi_0, t a_0, t \Psi_0)$, for small $t$, which is $O(t^2)$ away from being a Haydys monopole, as in \Cref{rem:Almost_KWm}. Since $\cC$ is an affine space over $\cV = (\Omega^1 (M, \g_P) \oplus \Omega^0 (M, \g_P))^{\oplus 2}$, we have that $T_{c_0} \cC \simeq \cV$. Let then $\delta c_0 = ( b_1, \phi, b_2, \psi) \in \cV$, and we search for a solution which is of the form $c = (m, v) = (m_0, t v_0) + \delta c_0$. As we are interested in solutions up to gauge equivalence only, it is convenient to work on the orthogonal complement of a slice of the gauge action. For that reason, we add the condition that $\delta c_0$ is orthogonal to the gauge slice passing through $(m_0, t v_0)$, which is equivalent to
\begin{equation}
	g_{c_0} (\delta c_0) = d_1^*(b_1, \phi) - \ast[a_0 \wedge \ast b_2] - [\Psi, \psi] = 0.
\end{equation}
We can further restrict the form of $\delta c_0$: write $\delta c_0 = (\delta m_0, \delta v_0)$, and require that both $\delta m_0$ and $\delta v_0$ are perpendicular to the kernel of $D$, that is to all tangent vector to the Bogomolny monopole moduli space at $m_0$. In other words, $\delta c_0 \in \Im (D^* \oplus D^*)$. We do not lose any generality due to this requirement in what follows. Indeed, a tangential component in $\delta m_0$ would just ``redefine'' $m_0$, and---even more clearly---a tangential component in $\delta v_0$ could be absorbed in $v_0$.

\subsection{The proof of \Cref{Mtheorem:Main_Haydys}}\label{subsec:Proof_Haydys}

\begin{proof}

From the discussion in the previous section, the gauge fixed Haydys monopole \cref{eq:Haydys_Mono_1,eq:Haydys_Mono_2,eq:Haydys_Mono_3} around $c_0 \in \cC$ is encoded in a map
\begin{equation}
	\hat{\kappa}_{c_0} : \cV \rightarrow \cR \oplus \Omega^0 (M, \g_P); \ \delta c_0 \mapsto (\kappa (c_0 + \delta c_0), g_{c_0} (\delta c_0)).  \label{eq:kappahat}
\end{equation}
Let $\cV_0 = \Omega^1 (M, \g_P) \oplus \Omega^0 (M, \g_P)$, and thus $\cV \simeq \cR \oplus \Omega^0 (M, \g_P) \simeq \cV_0^{\oplus 2}$. In this sense, the linearization of $\hat{\kappa}_{c_0}$ at a Bogomolny monopole $c_0 = (m_0, 0)$ is
\begin{equation}
	d\hat{\kappa}_{c_0} = 
	\begin{pmatrix}
		\cL & 0 \\
		0 & \cL
	\end{pmatrix},
\end{equation}
where $\cL$ is the gauge fixed linearization of the Bogomolny equation at $m_0$. Under \Cref{hypoth:maxsymbr}, $\cL$ is a continuous and surjective Fredholm operator (cf.  \cites{Taubes1983,K11,O16}), and thus so is $d\hat{\kappa}_{c_0}$. Hence we can conclude that $\cM_H$ is a smooth manifold around $\cM_B$, and moreover the normal bundle of $\cM_B$ is isomorphic to $T \cM_B$. In particular, $\dim (\cM_H) = 2 \dim (\cM_B)$. We construct below an explicit isomorphism.

Now let $c_0 = (m_0, t v_0)$, where again $v_0$ is a tangent vector of the Bogomolny monopole moduli space at $m_0$ with unit $L^2$-norm, and $t$ is to be specified later. Let $\delta c_0 = s \delta c$, where $\delta c = (\delta m, \delta v)$ and $s$ a {\em small} parameter also to be specified later. By \Cref{rem:Almost_KWm}, the gauge fixed Haydys monopole \cref{eq:Haydys_Mono_1,eq:Haydys_Mono_2,eq:Haydys_Mono_3} become $\hat{\kappa}_{c_0} (s \delta c) = 0$. Since $\hat{\kappa}_{c_0}$ is quadratic, the equation becomes
\begin{equation}
	 0 = \hat{\kappa}_{c_0} (s \ \delta c) = \hat{\kappa}_{c_0} (0) + s \ d \hat{\kappa}_{c_0} (\delta c) + s^2 \ Q_{c_0} (\delta c, \delta c),  \label{eq:Linearized_KWm}
\end{equation}
where $Q_{c_0}$ is the continuous quadratic remainder in the Taylor expansion of $\hat{\kappa}$ around $(c_0, 0)$, and there are no higher order terms. As noted in \Cref{rem:Almost_KWm}, we have
\begin{equation}
\hat{\kappa}_{c_0} (0) = (\kappa (m_0, t \ v_0), 0) = t^2 \ (\kappa (m_0, v_0), 0) = O (t^2 \ |v_0|^2).
\end{equation}
Short computation shows that in the direction $(b_1, \phi, b_2, \psi) \in \cV \simeq T_{c_0} \cC$, the linearization of $\kappa$ at $c_0 = (\nabla_0, \Phi_0, a_0, \Psi_0)$ is
\begin{equation}
	d\kappa_{c_0} (b_1,\phi,b_2,\psi) =
	\begin{pmatrix}
		d_2	(b_1, \phi) - \left( [a_0 \wedge b_2] + [\ast b_2, \Psi_0] + [\ast a_0, \psi] \right)  \\
		d_2 (b_2, \psi) + \left( [b_1 \wedge a_0] - [\ast b_1, \Psi_0] - [\ast a_0, \phi] \right)  \\
		d_1^* (b_2, \psi) - \left( \ast [b_1 \wedge \ast a_0] + [\phi, \Psi_0] \right)
	\end{pmatrix},
\end{equation}
which combined with the linearization of $g$, yields
\begin{equation}
	d \hat{\kappa}_{c_0} (\delta c) = d \hat{\kappa}_{c_0} (\delta m, \delta v) = (D \delta m, D \delta v) + t \ L_{c_0} (\delta c)
\end{equation}
for some continuous, linear remainder term $L_{c_0}$, which is algebraic (not a differential operator). Since $\delta c = (\delta m, \delta v) \in \Im (D^* \oplus D^*)$, we can write $(\delta m, \delta v) = (D^* u_1, D^* u_2) = \bD^* u$, where $\bD^*=D^* \oplus D^*$ and $u=(u_1, u_2) \in \cH_2^{\oplus 2}$. Thus we have
\begin{equation}
	d \hat{\kappa}_{c_0} (\delta c) = (D D^* u_1, D D^* u_2) + t \ L_{c_0} (\bD^* u ) = \bD \bD^* u + t \ L_{c_0} (\bD^* u).
\end{equation}
Hence the gauge fixed Haydys \cref{eq:Linearized_KWm}---now in terms of $u$---becomes
\begin{equation}
	 0 = t^2 \ (\kappa (m_0, v_0), 0) + s \ \bD \bD^* u + st \ L_{c_0} (\bD^* u) + s^2 \ Q_{c_0} (\bD^* u, \bD^* u),  \label{eq:Linearized_KWm2}
\end{equation}
By \Cref{cor:Green}, the operator $DD^*: \cH_2 \rightarrow \cH_0$ admits a continuous Green's operator $G$. Let $\bG = G \oplus G$, thus $\bG \circ \bD \bD^* = \id_{\cH_2^{\oplus 2}}$. Thus, applying $\bG$ to \cref{eq:Linearized_KWm2} yields
\begin{equation}\label{eq:greenedHaydys}
	\begin{aligned}
	0 	&= \bG \left( t^2 \ (\kappa (m_0, v_0), 0) + s \ \bD \bD^* u + st \ L_{c_0} (\bD^* u) + s^2 \ Q_{c_0} (\bD^* u, \bD^* u) \right)  \\
		&= t^2 \ \bG (\kappa (m_0, v_0), 0) + s \ u + st \ \bG (L_{c_0} (\bD^* u)) + s^2 \ \bG (Q_{c_0} (\bD^* u, \bD^* u)).
	\end{aligned}
\end{equation}
Note, that \cref{eq:greenedHaydys} can be rewritten as a fixed point equation on $u$ as
\begin{equation}
	u = F(u) = - \tfrac{t^2}{s} \ \bG (\kappa (m_0, v_0), 0) - t \ \bG (L_{c_0} (\bD^* u)) - s \ \bG (Q_{c_0} (\bD^* u, \bD^* u)).  \label{eq:Fixed_Point}
\end{equation}
In what follows, for each operator on a Banach space, $X$, let $\| X \|$ be its operator norm. In order to use the Banach Fixed Point Theorem, we now prove that for $t$ sufficiently small and $s$ chosen appropriately, $F$ is a contraction from $B_1 (0) \subset \cH_2^{\oplus 2}$ to itself. First, we prove it maps $B_1(0) \subset \cH_2^{\oplus 2}$ to itself. Indeed, using \Cref{lem:Multiplication_H_Spaces,cor:Fredholm} together with $\| v_0 \|_{\cH_0} = 1$, we obtain that for $u \in B_1(0) \subset \cH_2^{\oplus 2}$ 
\begin{align}
	\| F(u) \|_{\cH_2^{\oplus 2}} 	&\leqslant \tfrac{t^2}{s} \ \| \bG (\kappa (m_0, v_0), 0) \|_{\cH_2^{\oplus 2}} + t \ \| \bG (L_{c_0} (\bD^* u)) \|_{\cH_2^{\oplus 2}} + s \ \| \bG (Q_{c_0} (\bD^* u, \bD^* u)) \|_{\cH_2^{\oplus 2}}  \\
									&\leqslant C_{m_0} \ \| \bG \| \ \left( \tfrac{t^2}{s} + t \ \| \bD^* \| + s \ \| \bD^* \|^2 \right).
\end{align}
For a fixed $t > 0$, and varying, but positive $s$, the term in the parentheses is minimized when $s (t) = \tfrac{t}{\| \bD^* \|}$, in which case we get
\begin{equation}
	\| F(u) \|_{\cH_2^{\oplus 2}} \leqslant three C_{m_0} \| \bG \| \| \bD^* \| \ t \leqslant 6 C_{m_0} \| G \| \| D^* \| \ t.
\end{equation}
Hence, for $t \leqslant t_{\max} (m_0) = (6 C_{m_0} \| G \| \| D^* \|)^{-1}$, $F$ definitely maps the ball $B_1(0) \subset \cH_2^{\oplus 2}$ to itself. Now we show that in this case $F$ is also a contraction. Let $u, v \in B_1(0)$. Then, by \Cref{lem:Multiplication_H_Spaces,cor:Fredholm}, we have for $s = \tfrac{t}{\| \bD^* \|}$
	\begin{align}
		\| F (u) - F (v) \|_{\cH_2^{\oplus 2}}	&\leqslant t \ \| \bG (L_{c_0} (\bD^* (u-v))) \|_{\cH_2^{\oplus 2}}  \\
												&\quad + \tfrac{t}{\| \bD^* \|} \ \|(G \circ Q_{c_0})(\bD^* u, \bD^* u) - (G \circ Q_{c_0})(\bD^* v, \bD^* v) \|_{\cH_2^{\oplus 2}}  \\
												&\leqslant t \ C_{m_0} \ \| \bG \| \ \| \bD^* \| \ \| u - v \|_{\cH_2^{\oplus 2}}  \\
												&\quad + \tfrac{t}{\| \bD \|} \ \| \bG \| \ \|Q_{c_0} (\bD^* (u + v), \bD^*(u - v)) \|_{\cH_0^{\oplus 2}}  \\
												&\leqslant t \ C_{m_0} \ \| \bG \| \ \left( \| \bD^* \| \ \| u - v \|_{\cH_2^{\oplus 2}} + \tfrac{1}{\| \bD^* \|} \ \|\bD^*(u + v) \|_{\cH_1^{\oplus 2}} \| \bD^* (u - v)) \|_{\cH_1^{\oplus 2}} \right)  \\
												&\leqslant t \ C_{m_0} \| \bG \| \ \left( \| \bD^* \| \| u - v \|_{\cH_2^{\oplus 2}} + \tfrac{1}{\| \bD^* \|} \ \|\bD^* \|^2 \|u + v\|_{\cH_1^{\oplus 2}} \| u - v \|_{\cH_1^{\oplus 2}} \right)  \\
												&\leqslant t \ (3 C_{m_0} \| \bG \| \| \bD^* \|) \ \| u - v \|_{\cH_2^{\oplus 2}}  \\
												&\leqslant \tfrac{t}{t_{\max} (m_0)} \ \| u - v \|_{\cH_2^{\oplus 2}}.
	\end{align}
	Hence, if $t < t_{\max} (m_0)$ the hypotheses of the Banach Fixed Point Theorem apply, and so there is a unique solution to the fixed point \cref{eq:Fixed_Point}, which in turn provides a (unique) solution to the Haydys monopole \cref{eq:Haydys_Mono_1,eq:Haydys_Mono_2,eq:Haydys_Mono_3} of the form
	\begin{equation}
		(m, v) = (m_0 + s \ \delta m, t \ v_0 + s (t) \ \delta v) = \left( m_0 + t \ \tfrac{D^* u_1}{\| \bD^* \|}, t \ v_0 + t \ \tfrac{D^* u_2}{\| \bD^* \|} \right).
	\end{equation}
	In fact, $u$ can be approximated as
	\begin{equation}
		u = u (m_0, t \ v_0) = \lim\limits_{n \rightarrow \infty} F^n (0),
	\end{equation}
	where recall that the $(m_0, t \ v_0)$-dependence is encoded in $F$; see \cref{eq:Fixed_Point}. This concludes the proof of \Cref{Mtheorem:Main_Haydys}.
\end{proof}

\begin{remark}
	In our construction the neighborhood of the monopole moduli space we constructed is an open ball-bundle, with, a priori, varying radius $t_{\max} (m_0)$. In particular the normal bundle of $\cM_B \subset \cM_H$ is canonically isomorphic to the tangent bundle $T \cM_B$.
\end{remark}
	
\begin{remark}
	Note furthermore that the situation is similar to that of the Higgs bundle moduli space over a Riemann surface: every holomorphic bundle over a closed Riemann surface can be viewed as a Higgs bundle with vanishing Higgs field, and thus defines a submanifold of the Hitchin moduli space. Moreover the tangent bundle and normal bundles of this submanifold are isomorphic.\\
	We see the same picture with the Riemann surface replaced by $\mathbb{R}^3$, holomorphic bundles replaced by monopoles, and Higgs bundles replaced by Haydys monopoles.
\end{remark}

\begin{example}
	Let $\rG = \SU (2)$ and $\cM_B$ be the charge 1 moduli space; see \Cref{subsec:Dimension_Moduli_Space_Haydys} for definitions. In this case $\cM_B = \mathbb{R}^3 \times \mathbb{S}^1$ with its flat metric. Furthermore, there is an action of $\mathbb{R}^3$ by translations and of $\mathbb{S}^1$ by gauge transformations of the form $\exp(s\Phi)$, for $s\in \mathbb{R}$. This action extends to $\cM_H$ and using it we find that $\cM_H$ is a product Riemannian manifold of the form $\cM_H \simeq \mathbb{R}^3 \times \mathbb{S}^1 \times F$, where the fiber $F$ is a 4-dimensional hyperk\"ahler manifold along its (nonempty) smooth locus, $F_{\mathrm{smooth}}$. In particular, $t_{\max}$ is constant (and thus its infimum is positive), and the smooth locus of $\cM_B$ is $\mathbb{R}^3 \times \mathbb{S}^1 \times F_{\mathrm{smooth}}$.
\end{example}

\section{On the geometry of the Haydys monopole moduli space}\label{sec:Haydys_Geometry}

\subsection{Dimension of the Haydys moduli space}\label{subsec:Dimension_Moduli_Space_Haydys}

Let us consider the $\rG = \SU (2)$ case first. Let $(\nabla, \Phi)$ be a finite energy $\SU (2)$ Bogomolny monopole. By \Cref{hypoth:monodecay}, regarding $\ad(P_\infty)$ a real, oriented, rank-3 vector bundle over $\mathbb{S}_\infty^2$, the asymptotic Higgs field $\Phi_\infty \in \Gamma(\ad(P_\infty))$ is nonzero and $\nabla^\infty$-parallel. Hence $\Phi_\infty$ determines a map from $\mathbb{S}_\infty^2$ to the sphere in the Lie algebra $\mathfrak{su}(2) \cong \mathbb{R}^3$ and so is a map between two 2-spheres. Hence, it has a degree which is commonly called charge. Let us denote by $\cM_H^k$ the connected component of the moduli space of Haydys monopoles with real structure group $\SU (2)$ that contains the Bogomolny monopoles of charge $k$. We now prove that $\dim_{\mathbb{R}} (\cM_H^k) = 8 k$. Indeed, the linearization of the gauge fixed Haydys monopole map \eqref{eq:kappahat}, is the linear operator $D \oplus D$. As in the case of monopoles, from the analysis in \cite{Taubes1983}*{Proposition~9.2} follows that any tangent vector $v$ to $\cM_H^k$ at Haydys monopole $(\nabla, \Phi, a, \Psi)$ must be such that
\begin{equation}
	\| \nabla c\|^2 + \| [\Phi,c] \|^2 < \infty.
\end{equation}
Thus, $\dim_{\mathbb{R}} (\cM_H^k) = \dim(\ker_{\cH_1} (D \oplus D))= 2 \dim (\ker_{\cH_1} (D))$, which, in fact, coincides with the index of $D \oplus D : \cH_1 \rightarrow \cH_0$. Finally, for $k>0$, the dimension of $\ker_{\cH_1} (D)$ can be computed, as in \cite{Taubes1983}*{Proposition~9.1}, to be $4k$. Putting these together we conclude that 
\begin{equation}
	\dim_{\mathbb{R}} (\cM_H^k) = 8k.
\end{equation} 
Thus, we have constructed an open subset of $\cM_H^k$.

\begin{remark}
Also for all $\rG$ we construct an open subset of the connected component of $\cM_H$ containing $\cM_B$. This follows immediately from the fact that
\begin{equation}
	\dim_{\mathbb{R}}(\cM_H) = \dim(\ker_{\cH_1} (D \oplus D))= 2 \dim(\ker_{\cH_1} (D))=2 \dim_{\mathbb{R}}(\cM_B).
\end{equation}
For $\rG = \SU (N)$ the latest have been computed in \cite{SG18}*{Theorem~4.3.9} as being 4 times the sum of the magnetic weights.
\end{remark}

\subsection{Geometric structures on the Haydys monopole moduli space}\label{sec:Geometric_Strucutres}

\subsubsection{Linear model}\label{ss:Linear_Model}

Let $\mathfrak{g}$ be a real semisimple Lie algebra and consider the quaternionic vector space $V=\mathfrak{g} \otimes_{\mathbb{R}} \mathbb{H}$. Writing an element of $V$ as $A=A_0+iA_1+jA_2+kA_3$ may be equipped with a $\rG$-bi-invariant metric obtained by $\langle A, A' \rangle = \sum_{i=0}^3 \langle A_i, A'_i \rangle$, where in the rightmost term we use the Killing form on $\mathfrak{g}$. The quaternionic structure determines three symplectic structures $\omega_I$, $\omega_J$, $\omega_K$ with respect to which the adjoint action of $\rG$ is tri-Hamiltonian. The three moment maps associated with these respectively are 
\begin{equation}
	\nu_i(A) = [A_0,A_i] +[A_j,A_k]
\end{equation}
where $i = 1,2,3$ respectively and $(i,j,k)$ is a cyclic permutation of $(1,2,3)$.

This whole setup may be complexified by considering $\mathfrak{g}_{\mathbb{C}}$ rather than $\mathfrak{g}$. Then, we define
\begin{equation}
	V_{\mathbb{C}} = \mathfrak{g}_{\mathbb{C}} \otimes_{\mathbb{R}} \mathbb{H} = (\mathfrak{g} \otimes_{\mathbb{R}} \mathbb{H}) \oplus i (\mathfrak{g} \otimes_{\mathbb{R}} \mathbb{H}) \cong V \oplus V
\end{equation} 
and use the rightmost term to extend the inner product from $V$ to $V_{\mathbb{C}}$. Furthermore, we use this to consider the three quaternionic structures on $V_{\mathbb{C}}$ given by
\begin{equation}
	I_1 = \begin{pmatrix}
	0 & -1 \\
	1 & 0
	\end{pmatrix} \ , \ \ \ I_2=\begin{pmatrix}
	I & 0 \\
	0 & - I
	\end{pmatrix} \ , \ \ \ I_3=\begin{pmatrix}
	0 & I \\
	I & 0
	\end{pmatrix},
\end{equation}
and similarly for $J$ and $K$. Notice in particular that $I_1 = J_1 = K_1$ but that all these complex structures together do {\em not} form a compatible octonionic structure. Further notice that for example $I_2 \circ J_2 \neq K_2$ and $I_3 \circ J_3 \neq K_3$, in fact we have
\begin{equation}
	I_2 \circ J_2 \circ K_2 = I_3 \circ J_3 \circ K_3 = \diag(1,-1).
\end{equation}
Together with the inner product $\langle \cdot , \cdot \rangle$ these complex structures give rise to three sets of hyperk\"ahler structures with respect to which $\rG$ acts in an Hamiltonian fashion. The associated moment maps can be written as
\begin{equation}\label{eq:Moment_Map_1}
	\mu_{I_1}(A,B) = \sum_{i=0}^3 [A_i,B_i] 
\end{equation}
with $\mu_{J_1} = \mu_{K_1}$ being given by same formula, while
\begin{equation}\label{eq:Moment_Map_2}
	\mu_{I_2}(A,B) = \left( [A_0,A_1] +[A_2,A_3] \right) - \left( [B_0,B_1] + [B_2,B_3] \right),
\end{equation}
and $\mu_{J_2}$, $\mu_{K_2}$ similarly obtained by cyclic permutations of $(1,2,3)$. Finally, we have
\begin{equation}\label{eq:Moment_Map_3}
	\mu_{I_3}(A,B) = \left( [A_0,B_1] + [A_2,B_3] \right) - \left( [A_1,B_0] + [A_3,B_2] \right),
\end{equation}
with again $\mu_{J_3}$ and $\mu_{K_3}$ being obtained from a cyclic permutation of $(1,2,3)$.

\begin{remark}
	There is one further quite natural hyperk\"ahler structure on $V_{\mathbb{C}} \cong V \oplus V$ which is given by $\diag(I,I)$, $\diag(J,J)$ and $\diag(K,K)$. Using these we can still obtain the moment maps $\mu_{I_2}, \mu_{J_2}, \mu_{K_2}$ as follows. Instead of considering a Riemannian metric on $V_{\mathbb{C}}$ we use the indefinite pairing
	\begin{equation}
		b ((A,B), (A',B')) = \langle A, A' \rangle - \langle B, B' \rangle.
	\end{equation}
	Using it and the quaternionic structure above we define three symplectic forms with respect to which we can define moment maps very much in the same manner. These coincide with the moment maps $\mu_{I_2}, \mu_{J_2}, \mu_{K_2}$. 
\end{remark}

We now consider the joint moment maps 
\begin{equation}
	\mu_I = (\mu_{I_1}, \mu_{I_2}, \mu_{I_3}): V_{\mathbb{C}} \rightarrow \mathbb{R}^3,
\end{equation}
together with $\mu_J$ and $\mu_K$. As the complex structure $I_1=J_1=K_1$ is common to the three triples, it is the only one which immediately restricts to
\begin{equation}
	Q = \mu_I^{-1}(0) \cap \mu_J^{-1}(0) \cap \mu_K^{-1}(0).
\end{equation}
We must now check that all the other ones equally do restrict to $Q$. The first observation which is relevant for our analysis is the fact that the zero level set of moment maps $\mu_I,\mu_J,\mu_K$ are invariant under the map $\iota: V_{\mathbb{C}} \to V_{\mathbb{C}}$ given by $\iota(A,B)=(A,-B)$ as can be immediately seen from the \cref{eq:Moment_Map_1,eq:Moment_Map_2,eq:Moment_Map_3}. From this, we find the following.

\begin{lemma}\label{lem:Iota}
	Let
	\begin{equation}
		\Fix (\iota)= \lbrace (A,0) \in V_\mathbb{C} \ | \ A \in V \rbrace,
	\end{equation}
	be the fixed locus of the involution $\iota$. As $Q \subset V_{\mathbb{C}}$ is invariant under the map $\iota$ we find that $\Fix (\iota) \cap Q = \Fix (\iota|_Q)$.
\end{lemma}

We now use this to shortcut the proof of the following statement. 

\begin{proposition}\label{prop:Complex_Structures}
	Along all of $Q$, the tangent spaces to $Q$ are invariant under the complex structure $I_1=J_1=K_1$.\\
	Along $\Fix (\iota) \cap Q$, the tangent spaces to $Q$ are invariant under all the On the other hand structures $I_i, J_i, K_i$ for $i=1,2,3$. In other words, these complex structures preserve $TQ|_{\Fix (\iota|_Q)}$.
\end{proposition}
\begin{proof}
	For the first statement we simply note that by the standard K\"ahler reduction technique the tangent spaces to $\mu_{I_1}^{-1}(0)$ are preserved by $I_1$. Now, let $x \in \mu_I^{-1}(0) \cap \mu_J^{-1}(0) \cap \mu_K^{-1}(0)$, we show that $T_x = T_x(\mu_I^{-1}(0) \cap \mu_J^{-1}(0) \cap \mu_K^{-1}(0))$ is invariant by $I_1$. This follows immediately from the formulas $d\mu_{J_2}\circ I_1 = d\mu_{J_3}$ and $d\mu_{K_2}\circ I_1 = d\mu_{K_3}$ which can be obtained by direct inspection.
	
	As for the second part of the statement, instead of showing that for $x \in \Fix (\iota)\cap Q$ the tangent space $T_x$ is invariant by all these complex structures, we show that its orthogonal complement $T_x^{\perp}$ is itself invariant. To do this notice that $T_x = \cap_{\cI = I, J, K} \cap_{i = 1}^3 \ker(d\mu_{\cI_i})$ and so
	\begin{equation}
		T_x^{\perp}= \lbrace \nabla \mu_{I_i}, \nabla \mu_{J_i}, \nabla \mu_{K_i} \ , \ i =1,2,3 \rbrace.
	\end{equation}
	However, recall that for $L=I,J,K$ and $i=1,2,3$ we have that for any $\xi \in \mathfrak{g}$ 
	\begin{equation}
		d \langle \xi, \mu_{L_i}  \rangle (\cdot)  = \omega_{L_i}(\xi_\ast , \cdot )= \langle L_i \xi_\ast , \cdot \rangle,
	\end{equation}
	where $\xi_\ast$ denotes the vector field in $V_{\mathbb{C}}$ obtained via the infinitesimal action of $\xi \in \mathfrak{g}$. Thus, from this formula and the definition of the gradient we find that $\nabla \left( \langle \xi, \mu_{L_i}  \rangle \right) = L_i \xi_\ast$. Thus, we find
	\begin{equation}
		T_x^{\perp}= \lbrace I_i\xi_\ast, J_i\xi_\ast , K_i\xi_\ast \ , \ \text{for} \ i = 1,2,3 \ , \ \xi \in \mathfrak{g} \rbrace,
	\end{equation}
	and must show this is invariant under the complex structures $I_i, J_i , K_i$. On the other hand, from \Cref{lem:Iota} we find that $d \iota|_Q :TQ \to TQ$, and as $\iota|_{\Fix (\iota |_Q)}$ is the identity 
	\begin{equation}
		(d \iota) |_{\Fix (\iota |_Q)}:TQ|_{\Fix (\iota |_Q)} \to TQ|_{\Fix (\iota |_Q)},
	\end{equation}
	is a bundle map. Furthermore, as $\iota$ is a linear map on $V_{\mathbb{C}}$, then regarding $T_x$ and $T_x^\perp$ as a subspace of $V_{\mathbb{C}}$, we can identity its derivative $d (\iota) |_{\Fix (\iota |_Q)}$ with $\iota$ itself.
	
	Hence, in order to show that $T_x^\perp$ is invariant under the complex structures $I_i, J_i , K_i$. This would be immediate if these complex structures formed a closed algebra as that of the octonions for example. That is only true modulo $\iota$, indeed we have $ I_2 \circ J_2 = \iota \circ K_2$, $ I_3 \circ J_3 = \iota \circ K_3$ and similar formulas for other compositions.
\end{proof}

\begin{remark}
	Alternatively, we can explicitly check that for all $A \in V$
	\begin{equation}
		\bigcap_{L \in \lbrace I,J,K \rbrace} \bigcap_{i=1}^3 \ker(d\mu_{L_i})_{(A,0)},
	\end{equation}
	is preserved by all the complex structures.
\end{remark}

\begin{remark}
	Given that $I_1 = J_1 = K_1$ we have $\dim (Q)= \dim (\mathfrak{g})$ if $\mathfrak{g}$ is finite dimensional.
\end{remark}

Notice that the fixed point locus of the involution $\iota$ is given by $ \Fix (\iota) = V \oplus 0 \subset V_{\mathbb{C}}$ is a complex submanifold with respect to $I_2,J_2,K_2$ which restrict to $\Fix (\iota)$ as $I,J,K$ respectively, thus inducing an hyperk\"ahler structure there. On the other hand, $\Fix (\iota)$ is totally real with respect to all the remaining complex structures. In fact, it is complex-Lagrangian with respect to the respectively induced complex symplectic structure. This may be trivially checked by noticing that it is complex with respect to $\cI_2$ and Lagrangian with respect to $\omega_{\cI_3}+i\omega_{\cI_1}$ for all $\cI = I, J, K$. In summary we have the following.

\begin{lemma}\label{lem:Fixed_Points_1}
	Consider $V_{\mathbb{C}}$ equipped with the complex structure $I_1=J_1=K_1$ and the corresponding symplectic form $\omega$. The fixed locus $\Fix (\iota)$ is Lagrangian with respect to $\omega$.
	
	On the other hand, $\Fix (\iota)$ is complex with respect to any of the structures $I_2$, $J_2$, $K_2$. Furthermore, $I_2$, $J_2$, $K_2$, equip $\Fix (\iota)$ with an hyperk\"ahler structure. In fact, $\Fix (\iota)$ is a complex-Lagrangian submanifold of  $V_{\mathbb{C}}$ with respect to the hyperk\"ahler structures induced by $(I_2,I_3,I_1)$, $(J_2,J_3,J_1)$, $(K_2,K_3,K_1)$ on $V_{\mathbb{C}}$.
\end{lemma}

\begin{remark}
	In the terminology of \cite{KW07}, $\Fix (\iota)$ is an A-brane with respect to $(\omega_1,I_1=J_1=K_1)$ and an (ABA)-brane with respect to the three hyperk\"ahler structures on $V_{\mathbb{C}}$ induced by $(I_1,I_2,I_3)$, $(J_1,J_2,J_3)$, $(K_1,K_2,K_3)$.\\ 
	Notice that the involution $\iota$ is non (anti)-symplectic or (anti)-holomorphic, so the construction above is not contained in any standard use of involutions to find branes.
\end{remark}

This linear model serves as the model for a more general construction which we use to obtain some interesting geometric structures in the moduli space of solutions to the Haydys equation in \Cref{ss:proof_of_MT3}. In the next section, we briefly outline the finite dimensional curved version of that construction.

\subsubsection{Curved model}

Let $X$ be a smooth manifold equipped with three different hyperk\"ahler structures, compatible with the same metric $h$, as those of the previous section. This means that they are all compatible with the same Riemannian metric $h$ and the complex structures $(I_1,I_2,I_3)$, $(J_1,J_2,J_3)$, $(K_1,K_2,K_3)$ satisfy $I_1=J_1=K_1$, which we denote by $L_1$ for clarity, and
\begin{equation}
	I_2 \circ J_2  \circ K_2 = I_3 \circ J_3  \circ K_3,
\end{equation}
with each of these sides squaring to the identity.\footnote{It is unclear whether this condition could be dropped in some cases of interest.} 
Further, we suppose that there is a Lie group action $\rG$ which acts on $X$ in a tri-Hamiltonian fashion, with respect to all three hyperk\"ahler structures. In order to perform a meaningful reduction we seek to require the structure which makes \Cref{prop:Complex_Structures} work. For that we must imitate the existence of a map $\chi$, taking the role of $\iota$ in the proof of \Cref{prop:Complex_Structures}, which along its fixed loci closes the algebra of the complex structures and preserves the moment map equation. Such a bundle map must be $\rG$-invariant, so that it descends to the quotient. Finally, notice that the fixed loci of an isometric anti-holomorphic involution on a K\"ahler manifold is not only Lagrangian but a minimal submanifold. This is summarized as follows.

\begin{proposition}\label{prop:Quotient}
	Let $Q= \bigcap_{\cI = I, J, K} \bigcap_{i = 1,2,3} \mu_{\cI_i}^{-1} (0)$, then $(L_1,\omega_{L_1})$ descends to $Q/\rG$ which is a K\"ahler manifold of real dimension $\dim(X)-8\dim(\rG)$. 
	
	Suppose there is a smooth $G$-invariant $L_1$-anti-holomorphic isometric involution $\chi :X \to X$ satisfying $\chi (Q) \subset Q$. Then, $\chi$ descends to $\hat{\chi}:Q/\rG \to Q/\rG$ and its fixed locus $F$ is a minimal Lagrangian submanifold of $Q/\rG$.\\
	Furthermore, if along $\Fix (\chi)\cap Q$, the bundle maps 
	$$ \lbrace I_1,I_2,I_3,J_1,J_2,J_3,K_1,K_2,K_3 , (d\chi)|_{\Fix (\chi)\cap Q} \rbrace$$ 
	form a closed algebra. Then, all the hyperk\"ahler structures on $X$ determined by the metric and $(I_1,I_2,I_3)$, or $(J_1,J_2,J_3)$, or $(K_1,K_2,K_3)$ restricts to $\Fix (\chi)\cap Q$ and descend to $F \subset Q/\rG$ which then is an hyperk\"ahler manifold of real dimension $\frac{1}{2}(\dim(X)-8\dim(\rG))$ in three different ways.
\end{proposition}

\subsubsection{Proof of \Cref{Mtheorem:Geometric_Structures}}
\label{ss:proof_of_MT3}

\begin{proof}
By fixing a connection one can identify the space of connections on $P$ with $\Omega^1 (M,\g_P)$, and do the construction from \Cref{ss:Linear_Model} with the quaternionic Lie algebra $V$ replaced by $\Omega^1 (M,\g_P) \oplus \Omega^0 (M, \g_P)$, with $M = \mathbb{R}^3$. We now recall the flat hyperk\"ahler structure on $\Omega^1 (M, \g_P) \oplus \Omega^0 (M, \g_P)$. This is obtained by first fixing the usual $L^2$-metric
\begin{equation}
	h_B ( (\dot\nabla, \dot{\Phi}), ( \dot\nabla', \dot{\Phi}') ) = \int\limits_M \left( \langle \dot\nabla, \dot\nabla' \rangle + \langle  \dot{\Phi}, \dot{\Phi}' \rangle \right) \vol,
\end{equation}
also used to define the metric on the moduli space of Bogomolny monopoles. Then, we consider the complex structures $I_v$, parametrized by $v \in \mathbb{S}^2 \subset \mathbb{R}^3$ acting on $(c,\psi) \in \Omega^1 (M, \g_P) \oplus \Omega^0 (M, \g_P)$ as follows. Identify $(c,\psi)$ with $c + \psi dt \in \Omega^1 (M \times \mathbb{R}_t, \g_P)$, then use the identifications $M\times \mathbb{R}_t \cong v^{\perp} \oplus (v \mathbb{R} \times \mathbb{R}_t) \cong \mathbb{C}^2$ to define a complex structure on $M \times \mathbb{R}_t$ and define $I_v$ as its action by pullback on $\Omega^1 (M \times \mathbb{R}_t, \g_P)$. Using $e_1$, $e_2$, $e_3$ as the standard basis of $\mathbb{R}^3$, we write $I=I_{e_1}$, $J=I_{e_2}$, $K=I_{e_3}$.

Finally, we turn to our version of $V_{\mathbb{C}}$ which is the configuration space 
\begin{equation}
	\cC = (\cA \oplus \Omega^0 (M, \g_P) ) \oplus (\Omega^1 (M, \g_P) \oplus \Omega^0 (M, \g_P)),
\end{equation} 
equipped the constant coefficient metric $h$ given by 
\begin{equation}
	h ( (\dot\nabla, \dot{\Phi}, \dot{a}, \dot{\Psi}), ( \dot\nabla', \dot{\Phi}', \dot{a}', \dot{\Psi}') ) = h_B( (\dot\nabla, \dot{\Phi}), (\dot\nabla', \dot{\Phi}') ) + h_B( (\dot{a}, \dot{\Psi}), (\dot{a}', \dot{\Psi}') ),
\end{equation}
with $h_B$ as above. Then, we may equip $\cC$ with the three quaternionic structures $(I_1,I_2,I_3)$, or $(J_1,J_2,J_3)$, or $(K_1,K_2,K_3)$ as in \Cref{ss:Linear_Model}. As in there, the gauge group $\cG$ of automorphisms of $P$ acts on $\cC$ by conjugation and so we obtain the moment maps which for convenience we organize here as 
\begin{equation}
	(\mu_{I_i},\mu_{J_i},\mu_{K_i}) : \cC \to \mathbb{R}^3 \otimes \Omega^0 (M \g_P) \cong \Omega^1 (M \g_P),
\end{equation}
for $i=1,2,3$.\footnote{Note that here we are organizing the moment maps in a nonstandard way. Indeed, for any fixed $i$, the complex structures $I_i, J_i,K_i$ do not follow the quaternionic relations} A straightforward computation shows that the equation
\begin{equation}
	\ast F_\nabla - d_\nabla \Phi - \tfrac{1}{2} \ast [a \wedge a] + [a, \Psi] = 0, 
\end{equation}
can be identified with $(\mu_{I_2},\mu_{J_2},\mu_{K_2})=0$. In the same way, we have
\begin{equation}
	\ast d_\nabla a - d_\nabla \Psi - [a, \Phi] = 0,
\end{equation}
which can be identified with $(\mu_{I_3},\mu_{J_3},\mu_{K_3})=0$. Finally the last equation
\begin{equation}
	d_\nabla^* a + [\Psi, \Phi] = 0, 
\end{equation}
corresponds to $\mu_{I_1}=\mu_{J_1}=\mu_{K_1}=0$ which recall is only one equation as $I_1=J_1=K_1$.

Formally, the same argument as that we used in \Cref{prop:Complex_Structures}, shows that this common complex structure restricts to the locus $\cC_H \subset \cC$ cut out by the Haydys equations. Thus, the K\"ahler structure determined by $(h,I_1=J_1=K_1)$ descends to the quotient 
\begin{equation}
	\cM_H = \cC_H / \cG,
\end{equation} 
which can be identified with the moduli space of solutions to the Haydys equation. On $\cC$ we have an involution $\iota$ sending $c = (\nabla, \Phi, a, \Psi)$ to $\iota(c)=(\nabla, \Phi, - a, - \Psi)$ which trivially preserves $\cC_H$. Thus, by \Cref{lem:Fixed_Points_1}, $\Fix (\iota) \subset \cC$ and so is a complex Lagrangian submanifold of $\cC$ with respect to the whole three hyperk\"ahler structures $(L_2,L_3,L_1)$ where $L \in \lbrace I, J , K \rbrace$, i.e. an ABA-brane with respect to any of these. In particular, as stated in that Lemma $\Fix (\iota)$ is Lagrangian with respect to $I_1=J_1=K_1$. 

The points of $\Fix (\iota)$ correspond to those $c$ of the form $c = (\nabla, \Phi, 0, 0)$. In particular, for $c \in \Fix (\iota) \cap \cC_H$ we find that $(\nabla,\Phi)$ is actually a Bogomolny monopole. Furthermore, $\iota$ is a $\cG$-invariant $I_1=J_1=K_1$-antiholomorphic and isometric involution of $\cC$ and as stated in \Cref{prop:Quotient} it descends to the Haydys moduli space $\hat{\iota}:\cM_H \to \cM_H$. Its fixed loci is then the moduli subspace of Bogomolny monopoles 
$$\cM_B = \Fix (\hat{\iota}) = (\Fix (\iota) \cap \cC_H ) / \cG ,$$
which is then a minimal Lagrangian submanifold of $\cM_H$ to which the hyperk\"ahler structures $\lbrace L_1 , L_2, L_3 \rbrace$, for $L \in \lbrace I,J,K \rbrace$ descend and agree. In other words, it is a minimal Lagrangian A-brane which arises as a reduction of $\Fix (\iota)$ which is itself an ABA-brane of $\cC$ with respect to the hyperk\"ahler structures $\lbrace L_2 , L_3, L_1 \rbrace$. This is the main result of this section which we state as follows.

\begin{theorem}\label{thm:Fixed_Points_2}
	There is an infinite dimensional affine space $\cC $ equipped with three are three $\cG$-invariant hyperk\"ahler structures such that space of Haydys monopoles $\cC_H \subset \cC$ is cut out the zero set of all the corresponding moment maps.
	
	The moduli space of Haydys monopoles $\cM_H=\cC_H / \cG$ is a K\"ahler manifold which has the moduli space of Bogomolny monopoles $\cM_B$ as a minimal Lagrangian submanifold of $\cM_H$. Furthermore, the three different hyperk\"ahler structures on $\cC$ descend to $\cM_B$ equipping it with an hyperk\"ahler structure.
\end{theorem}

This finally implies \Cref{Mtheorem:Geometric_Structures}.
\end{proof}

\begin{remark}
	In the terminology of \cite{KW07}, $\cM_B$ is a minimal A-brane of the K\"ahler manifold $\cM_H$ which arises from reduction of an (ABA)-brane with respect to the three hyperk\"ahler structures on $\cC$ induced by all the different structures $(I_1,I_2,I_3)$, $(J_1,J_2,J_3)$, $(K_1,K_2,K_3)$. Furthermore, these do descend to $\cM_B$, which the induced structure from $(I_2,J_2,K_2)$ and the metric $h$ equips with an hyperk\"ahler structure. 
\end{remark}


\begin{remark}
	As in the case of the Bogomolny monopole moduli space $\cM_B$, the translations of $\mathbb{R}^3$ and the gauge transformations of the form $\exp(s\Phi)$ for $s\in \mathbb{R}$ give rise to flat directions in the moduli space $\cM_H$. Thus, as in the case of monopoles we may quotient out by these and consider centered configurations. 
\end{remark}

\bibliography{references}
\bibliographystyle{abstract}

\end{document}